\documentclass{article}
\usepackage[utf8]{inputenc}
\usepackage{amssymb}
\usepackage{amsthm}
\usepackage{amsmath,amscd}
\usepackage[mathscr]{euscript}
\usepackage[all]{xy}
\usepackage{lmodern}
\usepackage[T1]{fontenc}
\usepackage[textwidth=14cm,hcentering]{geometry}
\usepackage[linktocpage, colorlinks=true,linkcolor=red,citecolor=blue]{hyperref}
\usepackage{cleveref}
\usepackage{mathtools}
\usepackage{tikz}
\usepackage{stmaryrd}
\usetikzlibrary{cd}
\usetikzlibrary{arrows,positioning}

\usepackage{aliascnt}
\usepackage{tocloft}

\title{A monoidal Grothendieck construction for $\infty$-categories}
\author{ Maxime Ramzi}
\date{}

\newtheorem{thm}{Theorem}[section]
\newaliascnt{lm}{thm}  
\newtheorem{lm}[lm]{Lemma}
\aliascntresetthe{lm}
\Crefname{lm}{Lemma}{Lemmas}
\newaliascnt{prop}{thm}  
\newtheorem{prop}[prop]{Proposition}
\aliascntresetthe{prop}
\Crefname{prop}{Proposition}{Propositions}
\newaliascnt{cor}{thm}  
\newtheorem{cor}[cor]{Corollary}
\aliascntresetthe{cor}
\Crefname{cor}{Corollary}{Corollaries}

\newtheorem*{thm*}{Theorem}
\newtheorem*{cor*}{Corollary}
\newtheorem{thmx}{Theorem}

\newtheorem{corx}{Corollary}

\theoremstyle{definition}
\newaliascnt{defn}{thm}  
\newtheorem{defn}[defn]{Definition}
\aliascntresetthe{defn}
\Crefname{defn}{Definition}{Definitions}
\newaliascnt{cons}{thm}  

\aliascntresetthe{cons}
\Crefname{cons}{Construction}{Constructions}
\newaliascnt{nota}{thm}  

\aliascntresetthe{nota}
\Crefname{nota}{Notation}{Notations}
\newaliascnt{conv}{thm}  

\aliascntresetthe{conv}
\Crefname{conv}{Convention}{Conventions}
\newaliascnt{ex}{thm}  
\newtheorem{ex}[ex]{Example}
\aliascntresetthe{ex}
\Crefname{ex}{Example}{Examples}
\newaliascnt{rmk}{thm}  
\newtheorem{rmk}[rmk]{Remark}
\aliascntresetthe{rmk}
\Crefname{rmk}{Remark}{Remarks}
\newaliascnt{ques}{thm}  

\aliascntresetthe{ques}
\Crefname{ques}{Question}{Questions}
\newaliascnt{conj}{thm}  

\aliascntresetthe{conj}
\Crefname{conj}{Conjecture}{Conjectures}
\newaliascnt{warn}{thm}  
\newtheorem{warn}[warn]{Warning}
\aliascntresetthe{warn}
\Crefname{warn}{Warning}{Warnings}
\newaliascnt{obs}{thm}  
\newtheorem{obs}[obs]{Observation}
\aliascntresetthe{obs}
\Crefname{obs}{Observation}{Observations}
\newtheorem*{ques*}{Question}
\newtheorem*{rmk*}{Remark}
\newtheorem*{ex*}{Example}
\newtheorem*{defn*}{Definition}
\newaliascnt{recoll}{thm}  

\aliascntresetthe{recoll}
\Crefname{recoll}{Recollection}{Recollections}

\newcommand{\op}{^{\mathrm{op}}}
\newcommand{\cat}{\mathbf}
\newcommand{\on}{\operatorname}

\newcommand{\Fun}{\on{Fun}}
\newcommand{\map}{\on{map}}

\newcommand{\CMon}{\mathrm{CMon}}

\newcommand{\Ss}{\mathbf{An}}

\newcommand{\Fin}{\mathrm{Fin}}

\newcommand{\PrL}{\cat{Pr}^\mathrm{L}}

\newcommand{\Alg}{\mathrm{Alg}}
\newcommand{\CAlg}{\mathrm{CAlg}}
\newcommand{\Mod}{\cat{Mod}}

\newcommand{\colim}{\mathrm{colim}}

\newcommand{\Cat}{\cat{Cat} }
\newcommand{\C}{\cat C}
\newcommand{\D}{\cat D}
\newcommand{\E}{\cat E}
\newcommand{\Oo}{\mathcal O}
\newcommand{\LFib}{\cat{LFib}}
\newcommand{\coCart}{\cat{coCart}}

\newcommand{\RFib}{\cat{RFib}}
\newcommand{\Cart}{\cat{Cart}}

\newcommand{\pend}{\unskip\nobreak\hfill$\triangleleft$}

\DeclareFontFamily{U}{min}{}
\DeclareFontShape{U}{min}{m}{n}{<-> udmj30}{}
\providecommand{\keywords}[1]
{
  \small	
  \textbf{\textit{Keywords---}} #1
}

\begin{document}

\maketitle
\begin{abstract}
  We construct a monoidal version of Lurie's un/straightening equivalence. In more detail, for any symmetric monoidal $\infty$-category $\C$, we endow the $\infty$-category of coCartesian fibrations over $\C$ with a (naturally defined) symmetric monoidal structure, and prove that it is equivalent the Day convolution monoidal structure on the $\infty$-category of functors from $\C$ to $\Cat_\infty$. In fact, we do this over any $\infty$-operad by categorifying this statement and thereby proving a stronger statement about the functors that assign to an $\infty$-category $\C$ its category of coCartesian fibrations on the one hand, and its category of functors to $\Cat_\infty$ on the other hand. 
\end{abstract}

\newcommand\blfootnote[1]{%
  \begingroup
  \renewcommand\thefootnote{}\footnote{#1}%
  \addtocounter{footnote}{-1}%
  \endgroup
}
\blfootnote{
}
\tableofcontents
\newpage 
\section*{Introduction}
\addcontentsline{toc}{section}{Introduction}
\subsection*{Overview}
 \addcontentsline{toc}{subsection}{Overview}
The goal of this paper is to provide a monoidal version of the Grothendieck construction for $\infty$-categories, also known as straightening/unstraightening after \cite{HTT}.

Recall that this is a theorem which states that for any $\infty$-category $\C$, there is an equivalence of $\infty$-categories $\coCart_\C \simeq \Fun(\C,\Cat_\infty)$, where the left hand side is the (non-full) subcategory of $(\Cat_\infty)_{/\C}$ spanned by coCartesian fibrations, and morphisms of coCartesian fibrations between them (see \cite[Theorem 3.2.0.1]{HTT}).

If $\C$ is equipped with a symmetric monoidal structure, one can ask to what extent symmetric monoidal functors $\C\to \Cat_\infty$ correspond to ``symmetric monoidally coCartesian fibrations'', where $\Cat_\infty$ is equipped with the \emph{cartesian} symmetric monoidal structure. Our goal is to answer this question, inspired by \cite{MV} and the $1$-categorical version of our question therein.

We proceed in two steps: we start by doing a pointwise version of a monoidal Grothendieck construction which is well-known (it is essentially already contained in \cite{HA}, but we make it explicit for the convenience of the reader - see also \cite[Proposition A.2.1]{hinichrect}) and can be stated as follows, see  \Cref{thm : microcosm} (and see \Cref{section : micro} for a definition of $\Oo$-monoidally coCartesian fibrations):
\begin{thm}\label{thm:microintro}
    Let $\Oo$ be an $\infty$-operad and $\C$ an $\Oo$-monoidal $\infty$-category. There is an equivalence of $\infty$-categories $$\Fun^{lax-\Oo}(\C,\Cat)\simeq \coCart_\C^\Oo$$
between the $\infty$-category of lax $\Oo$-monoidal functors $\C\to \Cat$ and the $\infty$-category of $\Oo$-~monoidally coCartesian fibrations over $\C$ which, on underlying objects is the un/straightening equivalence.
\end{thm}

The second step is to apply this to a specific coCartesian fibration, namely the fibration $\coCart\to\Cat_\infty$ whose fiber over an $\infty$-category $\C$ is the $\infty$-category $\coCart_\C$. We show that when it is endowed with the cartesian symmetric monoidal structure, it fits into the context of the previous pointwise construction. This endows the functor $\C\mapsto \coCart_\C$ with a lax symmetric monoidal structure as a functor $\Cat_\infty\to \widehat{\Cat}_\infty$ (see  \Cref{defn : cancocart}), where the latter is the $\infty$-category of possibly large $\infty$-categories, equipped with its cartesian symmetric monoidal structure. 

Our main result compares this lax symmetric monoidal structure with the lax symmetric monoidal structure on $\C\mapsto \Fun(\C,\Cat_\infty)$ constructed from the universal property of presheaf $\infty$-categories, which amounts to a Day-convolution-style lax symmetric monoidal structure (see \Cref{defn : canday} - note that the \emph{functoriality} of the latter, that is, without symmetric monoidal structures is also induced from this universal property\footnote{Although this can be compared with other functorialities, see \cite{HHLN2}}).  We refer the reader to \Cref{section : meta} and the aforementioned definitions for details about these monoidal structures,  see also \Cref{rmk:content}.

With this in mind, our main result is the following:

\begin{thmx}\label{thm : metacosm}
The straightening/unstraightening equivalence $$\Fun(\C,\Cat_\infty)\simeq \coCart_\C$$ can be enhanced to an equivalence of lax symmetric monoidal functors between the canonical lax symmetric monoidal structures on $\C\mapsto \Fun(\C,\Cat_\infty)$ (\Cref{defn : canday}) and $\C\mapsto~\coCart_\C$ (\Cref{defn : cancocart}) respectively . 

In fact, this equivalence can be made $\Cat_\infty$-linear: it lifts to an equivalence of lax symmetric monoidal functors $\Cat_\infty \to \Mod_{\Cat_\infty}(\widehat{\Cat}_\infty)$, where $\Cat_\infty$ is viewed as a commutative algebra in $\widehat{\Cat}_\infty$, and this is why this module $\infty$-category makes sense. 
\end{thmx}
 Equivalences of lax symmetric monoidal functors induce equivalences of $\Oo$-algebras when applied to $\Oo$-algebras for any $\infty$-operad $\Oo$, so we deduce the following, which is known in the $1$-categorical case \cite{MV}, but as far as we are aware even this special case does not appear in the $\infty$-categorical literature (see \Cref{thm : macroprecise} for a precise description of the $\Oo$-$\infty$-operad structure of $\coCart_\C$)  :  
\begin{thmx}\label{thm  : macrocosm}
Let $\Oo$ be an $\infty$-operad and $\C$ an $\Oo$-monoidal $\infty$-category.

The straightening/unstraightening equivalence can be enhanced to a ($\Cat_\infty$-linear) $\Oo$-monoidal equivalence $$\Fun(\C,\Cat_\infty)\simeq \coCart_\C$$ where the domain has the Day convolution $\Oo$-monoidal structure, and the target $\Oo$-$\infty$-operad $(\coCart_\C)^\otimes$ is the full sub-$\Oo$-operad of $((\Cat_\infty)_{/\C})^\otimes$ consisting of those multi-morphisms that preserve coCartesian edges in each variable. 
\end{thmx}

\begin{rmk}
Taking this theorem for granted, any lax $\Oo$-monoidal functor $\C\to \Cat_\infty$ yields an $\Oo$-algebra in $\Fun(\C,\Cat_\infty)\simeq \coCart_\C$, and thus, by the description of the $\Oo$-$\infty$-operad structure of the latter, in $(\Cat_\infty)_{/\C}$. In the latter, these correspond to $\Oo$-monoidal functors $\D\to \C$. 

So the monoidal straightening/unstraightening equivalence yields an equivalence between $\Fun^{lax-\Oo}(\C,\Cat_\infty)$ and the subcategory of $\Alg_{\Oo}(\Cat_\infty)_{/\C}$ whose objects are the morphisms $\D\to \C$ that are coCartesian fibrations, and where the $\Oo$-operations preserves coCartesian edges; and morphisms are morphisms of $\Oo$-algebras over $\C$ that also preserve coCartesian edges. 
\pend\end{rmk}

\Cref{thm : metacosm} and \Cref{thm : macrocosm} have obvious analogues when replacing coCartesian fibrations with left fibrations, and $\Cat_\infty$-valued functors with $\Ss$-valued functors, where $\Ss$ is the $\infty$-category of anima\footnote{Anima or homotopy types or spaces or $\infty$-groupoids.}. These analogues are also immediate consequences of these theorems: 
\begin{corx}\label{cor : corlfib}
The anima-valued un/straightening equivalence can be made into a symmetric monoidal equivalence of functors $\Cat_\infty\to \PrL$ : $$\Fun(\C,\Ss)\simeq \LFib_\C$$
If $\Oo$ is an operad and $\C$ an $\Oo$-monoidal category, this specializes to an $\Oo$-monoidal equivalence between the Day-convolution and some $\Oo$-monoidal structure on $\LFib_\C$ analogous to the one described in \Cref{thm : macroprecise}. Taking $\Oo$-algebras therein specializes to an equivalence $$\Fun^{lax-\Oo}(\C,\Ss)\simeq \LFib^\Oo_\C$$
where the latter is defined analogously to \Cref{defn : monfib}. 
\end{corx}
Specializing further to the case where $\C= X$ is an anima, we obtain the following folklore result: 

\begin{corx}
Let $\Oo$ be an operad and $X$ an $\Oo$-algebra in anima. There is an equivalence of $\Oo$-monoidal categories $$\Fun(X,\Ss)\simeq \Ss_{/X}$$
where the left hand side has the Day convolution structure, while the right hand side has the comma category $\Oo$-monoidal structure. 
\end{corx}

 \subsection*{Relation to other work}
 \addcontentsline{toc}{subsection}{Relation to other work}
 Our approach is inspired by \cite{MV}, and specifically Lemma 3.9 therein, namely the observation that the coCartesian fibration $\coCart\to \Cat_\infty$ is cartesian, which drives our construction; see also \cite[Proposition 6.10]{haugchu}. 
We make some comments about the differences between our version and the $1$-categorical version of these statements from \cite{MV}: 
\begin{itemize}
    \item While \Cref{thm:microintro} has some content $1$-categorically, the formalism of $\infty$-operads and symmetric monoidal $\infty$-categories makes it essentially rigged to work - see the proof of \cite[Proposition A.2.1]{hinichrect} that we mentioned earlier. We only mention it in detail for convenience, and also to have access to the details in the later parts of the document; 
    \item In \cite{MV}, the authors also deal with the case where the base category $\C$ has no monoidal structure. In that case, one can compare algebras in the category of functors for the pointwise/cartesian monoidal structure  to algebras in the category of cocartesian fibrations for the pullback/cartesian symmetric monoidal structure. We do not know how to incorporate that in our setting in a way which is nontrivial \emph{and} does not rely on $(\infty,2)$-categorical technology which is not easy to access yet. 
\item For the same reason (lack of some needed $(\infty,2)$-categorical technology), we do not provide a statement about the symmetric monoidal $(\infty,2)$-functors $$\Fun(-,\Cat_\infty), \coCart_\bullet: \Cat_\infty\to \widehat{\Cat}_\infty$$ (nor do we construct them as $(\infty,2)$-functors), only their underlying symmetric monoidal $(\infty,1)$-functors. Such an extension would be very interesting, for example it would allow us to say more about naturality of the monoidal un/straightening equivalence in \emph{lax} symmetric monoidal functors, as opposed to strong ones (in the terminology of \cite{MV}, morphisms of pseudomonoids in the $(\infty,2)$-category $\Cat_\infty$, as opposed to morphisms of monoids in the underlying $(\infty,1)$-category). 
\end{itemize}

We now make a small remark, to be expanded upon in \Cref{rmk:content} after all precise definitions have been given. The contravariant functor $\Fun(-,\Cat_\infty)$ can be unstraightened to the cartesian fibration $(\widehat{\Cat_\infty})_{\sslash \Cat_\infty} \to \widehat{\Cat_\infty}$ which is in turn symmetric monoidal via cartesian products and also a coCartesian fibration, and can then be straightened into a covariant lax symmetric monoidal functor also deserving the name ``$\C\mapsto\Fun(\C,\Cat_\infty)$''. For \emph{this} functoriality, \Cref{thm : metacosm} is evident, since one can compare directly the fibrations $(\widehat{\Cat_\infty})_{\sslash \Cat_\infty}$ and $\coCart$, and they automatically have the same cartesian structure. Thus, \Cref{thm : metacosm} is about comparing the two covariant functorialities of $\Fun(-,\Cat_\infty)$: one via the universal property of presheaf categories, and one via ``passing to left adjoints'', in a \emph{symmetric monoidal way}. The comparison of the bare functorialities is already nontrivial, cf. \cite{HHLN2, yoneda}, and \emph{is used} in our proof - we do not obtain an independent argument.

Let us mention a last relation to other work, namely similar results in the case of left fibrations (and correspondingly functors to the $\infty$-category of anima, instead of $\infty$-categories). If we specialize our approach to \Cref{thm : metacosm} to this case, the key points are essentially contained in \cite[Section 6]{haugchu}, the ideas in which can be traced back to \cite{heine}. A version of \Cref{thm  : macrocosm} for anima and left fibrations can also be found stated in the literature, see e.g. \cite[Notation 3.1.2]{HLBrauer}. 

 \subsection*{Philosophy}
 \addcontentsline{toc}{subsection}{Philosophy}
 Before outlining the structure of this paper, we make a short comment about the philosophy relating these three results: viewing \Cref{thm : microcosm} as a special case of \Cref{thm  : macrocosm}, and the latter as a special case of \Cref{thm : metacosm} is an instance of the ``macrocosm principle'', see \cite[Section 2.2]{baezdolan} or \cite{nLabmicro} for a discussion. Here, the idea is that proving an equivalence of symmetric monoidal structures is much stronger than providing an equivalence of categories of algebras, and when the former is possible, it is a good explanation for the latter. Further, it is often actually \emph{simpler} to work in the more general setting because \emph{structures} often become \emph{properties} when one goes up the categorical ladder - this is what happens here as we go from general symmetric monoidal structures to \emph{cartesian} ones.  
 
 In this philosophy, \Cref{thm : microcosm} could be seen as a microcosm level statement, \Cref{thm  : macrocosm} as a macrocosm level statement, and finally \Cref{thm : metacosm} as a metacosm level statement.

 Let us now outline the contents of this paper. 
 
 \subsection*{Outline}
 \addcontentsline{toc}{subsection}{Outline}
In \Cref{section : pullback}, we prove a technical result possibly of independent interest, used in the proof of \Cref{thm : microcosm}.  

In \Cref{section : micro}, we apply the work of the previous section to obtain \Cref{thm : microcosm}. 

\Cref{section : pullback} and \Cref{section : micro} are mostly for background, intuition and possibly motivation. The reader who knows \cite[Proposition A.2.1]{hinichrect} (equivalently, \Cref{thm : microcosm}) can safely skip them. 

In \Cref{section : meta}, we bootstrap the microcosm version of the monoidal Grothendieck construction (\Cref{thm : microcosm}) up to prove the metacosm version (\Cref{thm : metacosm}). 

In \Cref{section : macro}, we deduce from the metacosm version (\Cref{thm : metacosm}) the macrocosm version of the statement (\Cref{thm  : macrocosm}) -- this is where we analyze the monoidal structure on $\coCart_\C$ and describe it concretely.

In \Cref{section : comp}, we study the compatibility of this functorial version with the poinwise version from \Cref{thm : microcosm}.  We are not able to show that they agree, but we outline a possible approach and explain how $(\infty,2)$-categorical technology would help us prove it.

Finally, in \Cref{section : straight}, we prove that coCartesian fibrations over an $\infty$-category $\C$ are closed under colimits in the corresponding slice category. This fact was important in an earlier version of this document, is of independent interest and is not as well-known as it ought to be.

\subsection*{Conventions}
 \addcontentsline{toc}{subsection}{Conventions}
We work in the framework of $\infty$-categories, as largely developed in the books \cite{HTT}, \cite{HA}. The key notational difference with these sources is that we use the word ``anima'' in place of ``space''\footnote{Other commonly used words include ``homotopy type'' and ``$\infty$-groupoid''. }, and let $\Ss$ denote the $\infty$-category{} of anima.

We use the word ``category'' to mean $\infty$-category, and specify $1$-category specifically for (nerves of) ordinary categories (we do not notationally distinguish between an ordinary category and its nerve). Accordingly, all the related notions that we use are to be interpreted in this sense : co/limits, functors, natural transformations, adjunctions etc.\ refer to their $\infty$-categorical version.

We use $\Cat$ to denote the category of small categories (in the above convention, this means the $\infty$-category of small $\infty$-categories), and $\widehat{\Cat}$ to denote the category of possibly large categories -- we will need in several instances to apply results proved for $\Cat$ to $\widehat{\Cat}$ : this can be formalized using the theory of universes, fixing one universe to be the universe of ``small'' sets; because everything we prove does not depend on the chosen universe, it applies one universe up. 

We will say something like ``by going one universe up'' to refer to this technique, and will not go into too many set-theoretic details, as they are not really relevant except in this regard. 

Similarly, we use $\coCart_\C$ to denote the (non-full) subcategory of $\Cat_{/\C}$ on coCartesian fibrations\footnote{Note that $\Cat$ denotes the $\infty$-category of $\infty$-categories, so here we mean coCartesian fibrations between $\infty$-categories rather than between quasicategories - in particular, these are closed under equivalences by design.} and coCartesian-edge-preserving functors between those; and $\widehat{\coCart}_\C$ for coCartesian fibrations from possibly large categories ($\C$ itself being allowed to be possibly large in this case - we will use it for $\C= \Cat$). Our general convention is that if $X\to Y$ is a coCartesian fibration and $f:y_0\to y_1$ is an edge in $Y$, $f_!:X_{y_0}\to X_{y_1}$ denotes the induced functor on fibers. We use $\PrL$ for the category of presentable categories.

We use the usual convention that $\to$ denotes arrows and $\mapsto$ denotes assignment:  we emphasize this here because both expressions such as $\C\to \Fun(\C\op,\Cat)$ and $\C\mapsto~\Fun(\C\op,\Cat)$ will appear, and we do not want the reader to get confused.

We use the definition of operads from \cite{HA} - we let $\Oo$ denote an operad, and $\Oo^\otimes$ denote the total category of the corresponding fibration $\Oo^\otimes\to \Fin_*$. We use $\Oo^\otimes_{\langle n\rangle}$ to denote the fiber over $\langle n\rangle$ of this fibration, and call $\Oo^\otimes_{\langle 1\rangle}$ the category of colours of $\Oo$.

We will say \emph{$\Oo$-monoidal functor} to refer to morphisms of $\Oo$-monoidal categories, and lax $\Oo$-monoidal functor for morphisms of their underlying $\Oo$-operads. If we want to stress that a morphism of $\Oo$-operads is a morphism of $\Oo$-monoidal categories, we will say that it is \emph{strong} $\Oo$-monoidal, but this means the same thing as $\Oo$-monoidal. When $\Oo=\mathrm{Comm}$ is the commutative operad, we will use the word (lax, strong) symmetric monoidal in place of (lax, strong) $\Oo$-monoidal 
\subsection*{Acknowledgements}
 \addcontentsline{toc}{subsection}{Acknowledgements}
I would like to thank Pierre Elis for asking the question from which this work originated and Shaul Barkan, Dustin Clausen, Bastiaan Cnossen for helpful comments, and especially to Rune Haugseng for many comments on an earlier draft. 

I would also like to thank my advisors, Jesper Grodal and Markus Land, for their constant support and interest in my projects. 

Finally, thanks to Fabian Hebestreit and to the anonymous referee for helpful comments that helped improve the exposition of the paper. 

This research was supported by the Danish National Research Foundation through the Copenhagen Centre for Geometry and Topology (DNRF151). A portion of this work was completed while the author was in residence at the Institut Mittag-Leffler in Djursholm, Sweden in 2022 as part of the program ‘Higher algebraic structures in algebra, topology and geometry’ supported by the Swedish Research Council under grant no. 2016-06596. 
\section{Pullbacks of monoidal structures}\label{section : pullback}
This section is essentially here for motivation, and for recollections - we claim no originality for any of the material here.

We will use the following notion: 
\begin{defn}
For an operad $\Oo$ and an $\Oo$-monoidal category $\D$, we call \emph{$\Oo$-operations}  the pushforwards of the form $\D^\otimes_X\xrightarrow{e_!}  \D^\otimes_Y$ for arrows $e : X\to Y$ in $\Oo^\otimes$. 
\pend\end{defn}
\begin{rmk}
For a general $Y\in \Oo^\otimes$, we can split it as $Y_1\oplus...\oplus Y_k$, with $Y_i\in \Oo^\otimes_{\langle 1\rangle}$ (using the notation from \cite[Remark 2.1.1.15]{HA}), and similarly split $X$, so that pushforwards along maps of the form $e : X\to Y$ can be written as products of pushforwards of the form $X_{j_1}\oplus ... \oplus X_{j_i})\to Y_i$. In particular, questions about these general pushforwards can often be reduced to the special case where $Y\in \Oo^\otimes_{\langle 1\rangle}$. 

For the same reason, they can often be reduced to the case of maps $X\to Y$ that are \emph{active}. 
\pend\end{rmk}
\begin{ex}\label{ex : commop}
When $\Oo= \mathrm{Comm}$, the $\Oo$-operation corresponding to the active morphism $\langle n\rangle \to \langle 1\rangle$ is exactly the tensor product $\D^n \simeq \D^\otimes_{\langle n \rangle}\to \D$. 
\pend\end{ex}

We wish to prove the following: 

\begin{prop}\label{prop : pullback}
Let $\Oo$ be an operad, $\C^\otimes,\D^\otimes,\E^\otimes$ be  $\Oo$-monoidal categories, and consider a diagram: $$\xymatrix{& \D^\otimes\ar[d] \\ \C^\otimes \ar[r] & \E^\otimes}$$
of operads over $\Oo$.
Suppose further that $\D^\otimes\to \E^\otimes$ is a morphism of $\Oo$-monoidal categories and that for every object $T\in \Oo^\otimes_{\langle 1\rangle}$, the induced map on fibers $\D^\otimes_T\to~ \E^\otimes_T$ is a coCartesian fibration, and that the coCartesian edges for these fibrations are preserved by the $\Oo$-operations.

In this situation, the pullback $\D^\otimes\times_{\E^\otimes}\C^\otimes \to \Oo^\otimes$ is an $\Oo$-monoidal category, and the morphism $\D^\otimes\times_{\E^\otimes}\C^\otimes \to \C^\otimes$ is a morphism of $\Oo$-monoidal categories. 

\end{prop}
\begin{rmk}\label{rmk : pullbackopd}
Note that in general, pullbacks are computed this way in the category of operads, that is, as ordinary pullbacks of the total categories - this will become apparent in the proof of \Cref{prop : pullback}. Here we are saying that this pullback in the category of operads over $\Oo$ is an $\Oo$-monoidal category, and that one of the functors is strong $\Oo$-monoidal.

In particular, the map from the pullback to $\D^\otimes$ is a morphism of $\Oo$-operads, i.e. a lax $\Oo$-monoidal functor of $\Oo$-monoidal categories.
\pend\end{rmk}
\begin{rmk}\label{rmk : spelloutoop}
In the hypothesis that ``coCartesian edges are preserved by the $\Oo$-operations'', we abuse terminology by identifying $\D^\otimes_X$ with $\prod_{i=1}^n\D^\otimes_{X_i}$, where $X= X_1\oplus ... \oplus X_n$ and $X_i\in \Oo^\otimes_{\langle 1\rangle}$. So this hypothesis states that if $X\to Y$ is a morphism in $\Oo^\otimes$, the corresponding $\Oo$-operation $\prod_{i=1}^n\D^\otimes_{X_i} \simeq \D^\otimes_X\to \D^\otimes_Y\simeq \prod_{j=1}^m \D^\otimes_{Y_j}$ preserves coCartesian edges. 

It is easy to see that it suffices to have this assumption when $m=1$ and $X\to Y$ is active. 
\pend\end{rmk}
Let us state the special case where $\Oo=\mathrm{Comm}$ for motivation and intuition:
\begin{cor}\label{cor : symmon}
Let $\C^\otimes,\D^\otimes,\E^\otimes$ be symmetric monoidal categories, and consider a diagram: $$\xymatrix{& \D^\otimes\ar[d] \\ \C^\otimes \ar[r] & \E^\otimes}$$
of lax symmetric monoidal functors.
Suppose further that $\D^\otimes\to \E^\otimes$ is \emph{strong} symmetric monoidal and that the underlying functor  $\D\to \E$ is a coCartesian fibration, in which the coCartesian edges are preserved  under tensor products.

In this situation, the pullback $\D^\otimes\times_{\E^\otimes}\C^\otimes$ is a symmetric monoidal category, and the morphism $\D^\otimes\times_{\E^\otimes}\C^\otimes \to \C^\otimes$ is a \emph{strong} symmetric monoidal functor. 
\end{cor}
\begin{rmk}
In this special case, by \Cref{ex : commop} and \Cref{rmk : spelloutoop}, the assumption is that tensor products of coCartesian edges in $\D$ remain coCartesian. 
\pend\end{rmk}
Let us make a small observation: in the setting of \Cref{cor : symmon}, if the functor $\C\to~\D$ were also symmetric monoidal, then the result would be simpler. Indeed, symmetric monoidal categories are equivalent to commutative monoids in $\Cat$, and the forgetful functor $\CMon(\Cat)\to~\Cat$ preserves and creates all limits. Expanding on this argument also yields a uniqueness statement in this case.

The interesting aspect of this statement is the fact that $p:\C\to\E$ is allowed to be merely lax symmetric monoidal. It is this relaxation that forces us to impose further conditions on $q:\D\to \E$. 

Let us now give an intuition for this special case : for this, we deal with everything in a naive way. Suppose we want to define a tensor product on $\C\times_\E\D$. A natural guess is to put $(c_0,e_0)\otimes (c_1,e_1) := (c_0\otimes c_1, e_0\otimes e_1)$. This does not work because this may no longer ``be in the pullback'' : indeed $q(e_0\otimes e_1) \simeq q(e_0)\otimes q(e_1)\simeq p(c_0)\otimes p(e_1)$, but this need not be equivalent to $p(c_0\otimes c_1)$ as $p$ is only lax symmetric monoidal. 

To remedy this, we observe that we nonetheless have a map $$q(e_0\otimes e_1) \simeq q(e_0)\otimes q(e_1)\simeq p(c_0)\otimes p(c_1)\to p(c_0\otimes c_1).$$ The coCartesian-ness assumption allows us to lift this to a map $e_0\otimes e_1\to \tilde e$ with $q(\tilde e)=~p(c_0\otimes~c_1)$. In particular, we may put $(c_0,e_0)\otimes (c_1,e_1) := (c_0\otimes c_1, \tilde e)$. 

When one tries to prove that this construction ``is'' associative, one will need to use that coCartesian edges are closed under tensor products. 

We now move on to the actual proof. We start by recalling a general lemma that allows one to ``glue'' coCartesian fibrations (this will be useful again later in the paper):

\begin{lm}{\cite[Lemma A.1.8]{haugseng2022shifted}}\label{lm : cartfiber}
Let $X,Y,S$ be categories and $p: X\to Y$ be a morphism of coCartesian fibrations over $S$, and suppose for each $s\in S$, the induced morphism on fibers $p_s : X_s\to Y_s$ is a coCartesian fibration, and that for each $e: s\to t\in S$, $e_! : X_s\to X_t$ sends $p_s$-coCartesian morphisms to $p_t$-coCartesian morphisms.  

Then $p$ is a coCartesian fibration, and $p_s$-coCartesian edges in $X_s$ are mapped to $p$-coCartesian edges in $X$ along $X_s\to X$. 
\end{lm}

We use this to prove the following well-known lemma : 
\begin{lm}\label{lm : glob}
Let $\Oo$ be an operad, $p^\otimes: \D^\otimes\to \Oo^\otimes$ and $q^\otimes: \E^\otimes\to \Oo^\otimes$ be $\Oo$-monoidal categories, and $\pi^\otimes : \D^\otimes\to \E^\otimes$ a morphism of $\Oo$-monoidal categories. 

Suppose that for any object $T\in\Oo^\otimes_{\langle 1\rangle}$, the induced map on fibers $\D^\otimes_T\to \E^\otimes_T$ is a coCartesian fibration, and suppose that coCartesian edges are preserved by the $\Oo$-operations.

Then the functor $\pi^\otimes$ is a coCartesian fibration. 
\end{lm}

\begin{proof}
We first observe that for any object $T\in \Oo^\otimes$, the induced morphism on fibers $\D^\otimes_T\to~\E^\otimes_T$ is a coCartesian fibration. Indeed, if $T\in \Oo^\otimes_{\langle 1\rangle}$, this is our assumption. Now, if $T\in \Oo^\otimes_{\langle n\rangle }$, then using the notation from \cite[Remark 2.1.1.15]{HA}, we can write $T= T_1\oplus ... \oplus T_n$ with $T_i\in \Oo^\otimes_{\langle 1\rangle}$, and by the definition of $\Oo$-monoidal category,  we have $\D^\otimes_T\simeq \prod_{i=1}^n \D^\otimes_{T_i}$, and similarly for $\E^\otimes$, in a compatible way. Because products of coCartesian fibrations are coCartesian, this proves the claim.

We are now in the situation of \Cref{lm : cartfiber}:  the assumption that $\Oo$-operations preserve coCartesian edges is exactly the assumption that $e_! : X_s\to X_t$ sends $p_s$-coCartesian edges to $p_t$-coCartesian edges (cf. \Cref{rmk : spelloutoop}), in the notation of that lemma. 
\end{proof}
\begin{defn}\label{defn : monfib}
We call such a functor $\D^\otimes\to \E^\otimes$, or abusively $\D\to \E$, an $\Oo$-monoidally coCartesian fibration; and we let $\coCart_\E^\Oo$ denote the category of $\Oo$-monoidally coCartesian fibrations over $\E$ : this is the full subcategory of $\coCart_{\E^\otimes}$ spanned by such $\D^\otimes$.\footnote{These are exactly $\E^\otimes$-monoidal categories, but we want to emphasize the focus on $\E$ as an $\Oo$-monoidal category here.}
\pend\end{defn}
We now have a second general lemma about coCartesian fibrations: 
\begin{lm}\label{lm : fibs}
Consider a diagram:
\[\begin{tikzcd}
	\C & \E & \D \\
	& S
	\arrow["p"', from=1-1, to=2-2]
	\arrow["f", from=1-1, to=1-2]
	\arrow["g"', from=1-3, to=1-2]
	\arrow["q", from=1-3, to=2-2]
	\arrow["r"{description}, from=1-2, to=2-2]
\end{tikzcd}\]
of categories, where $p: \C\to S, q: \D\to S, r: \E\to S$ are coCartesian fibrations, $f:~\C\to~\E, g: \D\to \E$ morphisms over $S$, and assume that $g$ sends $q$-coCartesian morphisms to $r$-coCartesian ones, and is a coCartesian fibration. 

Then $\pi : \C\times_\E\D\to S$ is a coCartesian fibration, and the projection $\C\times_\E\D\to \C$ preserves coCartesian edges. 
\end{lm}
\begin{proof}
First note that $\mathrm{pr}:\C\times_\E\D\to \C$ is a coCartesian fibration, as it is pulled back from one. Therefore, the composite $\C\times_\E\D\to \C\to S$ is also a coCartesian fibration.

The claim about coCartesian edges follows from \cite[Proposition 2.4.1.3. (3)]{HTT} : indeed, edges that are $\mathrm{pr}$-coCartesian lying over edges in $\C$ that are $p$-coCartesian are themselves $\pi$-coCartesian, and because $\C\times_\E\D\to \C$ and $\C\to S$ are coCartesian fibrations, there is a sufficient supply of those, hence they are the only $\pi$-coCartesian edges in $\C\times_\E\D$. In particular they map to $p$-coCartesian edges in $\C$.  
\end{proof}
We are now ready to prove  the main result of this section, i.e. \Cref{prop : pullback}.
\begin{proof}[Proof of \Cref{prop : pullback}]
By \Cref{lm : glob}, $\D^\otimes\to \E^\otimes$ is a coCartesian fibration. 

Therefore, by \Cref{lm : fibs} and the assumption that $\D^\otimes\to \E^\otimes$ is a morphism of $\Oo$-monoidal categories, it follows that $\C^\otimes\times_{\E^\otimes}\D^\otimes\to \Oo^\otimes$ is a coCartesian fibration, and that the projection to $\C^\otimes$ preserves coCartesian edges. To simplify notation, we will call the pullback $K^\otimes$. 

According to \cite[Definition 2.1.2.13]{HA}, to conclude we now have to show that for every $T\in \Oo^\otimes$, decomposing as $T= T_1\oplus ... \oplus T_n$, the inert morphisms $T\to T_i$ induce an equivalence $K^\otimes_T \simeq \prod_{i=1}^n K^\otimes_{T_i}$. 

But this follows from the same claim for $\C^\otimes,\D^\otimes,\E^\otimes$ and the fact that $\C^\otimes\to \E^\otimes$ preserves coCartesian edges over inert edges (even if it doesn't preserve all coCartesian edges). 

Furthermore, \Cref{lm : fibs} also shows that $K^\otimes\to \C^\otimes$ preserves coCartesian edges, so in particular it is a morphism of $\Oo$-monoidal categories. 
\end{proof}
We note that the third- and second-to-last paragraphs, where we discuss coCartesian edges over inert edges constitute a proof of \Cref{rmk : pullbackopd}. 

We record the following description of algebras in the pullback, which follows from \Cref{rmk : pullbackopd} - it will not be used in the rest of this note but is of independent interest: 
\begin{cor}
Suppose $\D^\otimes\to \E^\otimes$ is as above, and $f: (\Oo')^\otimes\to\E^\otimes$ is an $\Oo'$-algebra in $\E^\otimes$ for some operad map $\Oo'\otimes\to \Oo$. Then the fiber of $\Alg_{\Oo'/\Oo}(\D)\to\Alg_{\Oo'/\Oo}(\E)$ at $f$ is equivalent to $\Alg_{\Oo'/\Oo}(\Oo'\times_\E \D)$
\end{cor}

\begin{rmk}
See \cite[Construction IV.2.1]{NS} for a related construction, where the authors take a pullback of one strong monoidal functor $F: \C\to \D$, and a lax one, $G: \C\to \D$, along the projection $\Fun(\Delta^1,\D)\to \D\times\D$. This construction does not appear to be exactly a special case of the one outlined above, because $\Fun(\Delta^1,\D)\to \D\times \D$ is not a coCartesian fibration (it is ``cartesian in the first variable, coCartesian in the second''\footnote{They are called bifibrations in \cite{HTT}, and are special cases of the (curved) orthofibrations of \cite{HHLN,HHLN2}.}, which is why they need $F$ to be strong monoidal). Is there a common generalization of both constructions ? 

The two constructions have a common special case, the one where $F$ is constant equal to the unit of $\D$, where the pullback described above (a ``lax equalizer'') is the pullback of $\D_{\mathbf 1/}\to \D$ along $G$. 
\pend\end{rmk}

\section{The microcosmic monoidal Grothendieck construction}\label{section : micro}
In this section, we describe a construction of the equivalence between lax monoidal functors and monoidally coCartesian fibrations. It is essentially already contained in \cite{HA}, and is explicitly proved in \cite[Proposition A.2.1]{hinichrect}, but we spell it out for the convenience of the reader. 

Recall its statement: 
\begin{thm}\label{thm : microcosm}
Let $\Oo$ be an operad and $\C$ an $\Oo$-monoidal category. There is an equivalence of categories $$\Fun^{lax-\Oo}(\C,\Cat)\simeq \coCart_\C^\Oo$$
between the category of lax $\Oo$-monoidal functors $\C\to \Cat$ and the category of $\Oo$-monoidally coCartesian fibrations over $\C$ which, on underlying objects is the un/straightening equivalence.
\end{thm}

We will spell out the constructions of the functors and sketch a proof that they form an inverse equivalence, we encourage the reader to consult \cite[Proposition A.2.1]{hinichrect} for details.

The observation for this result is that the straightening/unstraightening equivalence is obtained by pulling back along a specific coCartesian fibration over $\Cat$.

For simplicity of understanding, let us begin with the case of left fibrations. In this case, the universal left fibration is $\Ss_{*/}\to \Ss$. Further, the Grothendieck construction, or unstraightening of a functor $F: \C\to \Ss$ is given by the pullback $\C\times_\Ss \Ss_{*/}$, so one can describe objects of the Grothendieck construction as pairs $(c,x)$ where $c\in \C$ and $x\in F(c)$.

As the functor $\Ss_{*/}\to\Ss$ preserves finite products, it is canonically symmetric monoidal if we view both categories as \emph{cartesian} symmetric monoidal categories. In particular, if $F:~\C\to~\Ss$ is any lax symmetric monoidal functor, its Grothendieck construction $\C\times_{\Ss}~\Ss_{*/}$ admits a canonical symmetric monoidal structure for which the projection to $\C$ is symmetric monoidal by \Cref{prop : pullback}. The description from \Cref{section : pullback} shows that this construction matches the naive version one could think of : letting $\mu_{x,y}:~F(x)\times~F(y)\to~F(x\otimes~y)$ denote the natural transformation coming from the lax symmetric monoidal structure on $F$, the tensor product in this monoidal structure is $(c_0,x_0)\otimes (c_1,x_1) := (c_0\otimes c_1, \mu_{c_0,c_1}(x_0,x_1))$. 

The case of $\Cat$ is not actually different, it is simply a bit more complicated as the universal coCartesian fibration is a bit harder to describe in $1$-categorical terms : a reasonable notation for it is $\Cat_{*\sslash}$ - an informal description of this is as follows : objects are categories with a distinguished object $(D,d)$, and a morphism $(D,d)\to (E,e)$ is a functor $f: D\to E$ together with a morphism $f(d)\to e$.

See \cite[Definition 3.3.11]{land} and the subsequent remarks for more information; and see \cite[Remark 7.22]{HHLN2}, \cite[Corollary 6.4]{haugchu} for a formal description of the universal coCartesian fibration. 

The point is that this category still has finite products, and they are still preserved by the functor $\Cat_{*\sslash}\to \Cat$, so that we can still view it as a symmetric monoidal functor between cartesian symmetric monoidal categories.

Furthermore, the coCartesian edges are those of the form $(D,d)\to (E,e)$ with $f: D\to E$ and where $f(d)\to e$ is an equivalence. These are clearly preserved under products, and so $\Cat_{*\sslash}\to \Cat$ satisfies the assumptions that we put on $\D\to \E$ in \Cref{prop : pullback}. In particular we can prove \Cref{thm : microcosm}:
\begin{proof}[Sketch of proof of \Cref{thm : microcosm}]
By \cite[Proposition 2.4.1.7]{HA}, lax $\Oo$-monoidal functors $\C\to \Cat$ are equivalently functors $\C^{\otimes}\to \Cat$ (as opposed to $\C^{\otimes}\to\Cat^\times$) satisfying an extra condition, namely that of being a ``lax cartesian structure'' in the sense of \cite[Definition 2.4.1.1]{HA}. 

So we find equivalences $$\Fun^{lax-\Oo}(\C,\Cat)\simeq \Fun^{lax-cart}(\C^{\otimes},\Cat)\subset \Fun(\C^{\otimes},\Cat)\simeq \coCart_{\C^{\otimes}}$$
where the last equivalence is the unstraightening equivalence. 

We claim two things about the composite fully faithful embedding: firstly, that its image is exactly the category of $\Oo$-monoidally coCartesian fibrations $\coCart_{\C}^\Oo\subset \coCart_{\C^\otimes}$ (which will conclude the proof); and second that on objects it can be described as $$(f:~\C\to~\Cat)\mapsto~(\C\times_{\Cat}~\Cat_{*\sslash}\to~\C)$$ with the $\Oo$-monoidal structure described earlier.

The first verification amounts to the verification that $\coCart_\C^\Oo$ is indeed equivalently the category of $\C^\otimes$-monoidal categories in the sense of \cite[Definition 2.1.2.13]{HA} (which will rely on \cite[Proposition 2.1.2.12]{HA}), and is done by Hinich in \cite{hinichrect}. 

The second verification is strictly speaking not needed for the proof but we make it nonetheless to connect with the pullback description from above, for intuition and concreteness. Essentially by design, the construction described in the previous paragraph fits in a pullback square: 
\[\begin{tikzcd}
	P_1 & {(\Cat_{*\sslash})^\times} \\
	{\C^\otimes} & {\Cat^\times} & {}
	\arrow[from=1-1, to=1-2]
	\arrow[from=1-1, to=2-1]
	\arrow[from=1-2, to=2-2]
	\arrow["f"', from=2-1, to=2-2]
\end{tikzcd}\]
while the one coming from the above string of equivalences fits in a pullback square: 
\[\begin{tikzcd}
	{P_2} & {\Cat_{*\sslash}} \\
	{\C^\otimes} & \Cat & {}
	\arrow[from=1-1, to=1-2]
	\arrow[from=1-1, to=2-1]
	\arrow[from=1-2, to=2-2]
	\arrow["f"', from=2-1, to=2-2]
\end{tikzcd}\]

Thus to compare the two, it suffices to construct a pullback square: 
\[\begin{tikzcd}
	{(\Cat_{*\sslash})^\times} & {\Cat_{*\sslash}} \\
	{\Cat^\times} & \Cat & {}
	\arrow[from=1-1, to=1-2]
	\arrow[from=1-1, to=2-1]
	\arrow[from=1-2, to=2-2]
	\arrow[from=2-1, to=2-2]
\end{tikzcd}\]
We now observe that the construction of the map $D^\times\to D$ from \cite[Proposition 2.4.1.5]{HA} is natural in the category with finite products $D$, so that we have such a naturality square. Thus we have a map from $(\Cat_{*\sslash})^\times$ to the pullback over $\Cat^\times$, encoding a map of symmetric monoidal coCartesian fibrations over $\Cat$, which clearly induces an equivalence on underlying categories, and is thus an equivalence as well. 




\end{proof}


\begin{rmk}
   From the construction, we find that strong $\Oo$-monoidal functors $\C\to\Cat$ correspond to the $\Oo$-monoidally coCartesian fibrations $\D\to\C$ satisfying the following property: for any active morphism $X\to Y$ in $\Oo^\otimes$ and any coCartesian lift thereof $c_X\to c_Y$ in $\C^\otimes$, the induced functor $\D_{c_X}\to \D_{c_Y}$ is an equivalence. 

   When $\Oo=\mathrm{Comm}$, this is slightly more concrete: the active morphisms are (generated by) those of the form $\mu_n :\langle n\rangle \to \langle 1\rangle$, and so the corresponding coCartesian lifts in $\C^\otimes$ are those of the form $(c_1,...,c_n)\to c_1\otimes ...\otimes c_n$. Since $\D^{\otimes}\to \C^\otimes$ is a symmetric monoidally coCartesian fibration, $D^{\otimes}_{(c_1,...,c_n)}\simeq \prod_{i=1}^n \D_{c_i}$ and the induced functor is $\prod_{i=1}^n \D_{c_i}\to \D_{c_1\otimes...\otimes c_n}, (d_1,...,d_n)\mapsto d_1\otimes ... \otimes d_n$. 
\pend\end{rmk}
\section{The metacosmic monoidal Grothendieck construction}\label{section : meta}
In this section, we prove \Cref{thm : metacosm}, our main result. As indicated in the introduction, our proof is inspired by the one in \cite{MV}, specifically lemma 3.9 therein, which is related to \cite[Proposition 6.10]{haugchu}, and is essentially the key point of our approach. We will also need to bootstrap from the microcosmic version.

We start by recalling the following, see \cite[Corollary 6.4]{haugchu}, where the notation is slightly different, they use $\mathrm{LSl}$ for what we denote here by $\LFib^{rep}$, namely we let $\LFib \subset~\Fun(\Delta^1,\Cat)$ denote the full subcategory spanned by left fibrations, and $\LFib^{rep}\subset\LFib$ the full subcategory spanned by \emph{representable} left fibrations, i.e. those equivalent to one of the form $\C_{x/}\to \C$.  

\begin{prop}
The restriction $ev_1: \LFib^{rep}\to \Cat$ of the evaluation functor is a coCartesian fibration. An edge therein 
$$\xymatrix{\C_{x/} \ar[r]\ar[d] & \D_{y/}\ar[d]\\ \C\ar[r]^f & \D}$$
is $ev_1$-coCartesian if and only if the top arrow preserves initial objects, i.e. if the canonical map $y\to f(x)$ induced by this square is an equivalence.

Furthermore, the functor $\Cat\to \Cat$ classified by $ev_1: \LFib^{rep}\to \Cat$ is exactly $\C\mapsto \C\op$.
\end{prop}
\begin{rmk}\label{rmk : naturalyoneda}
The proof in \cite{haugchu} relies on ``the naturality of the Yoneda embedding''. The fact that the Yoneda embedding is natural for the appropriate functoriality was proved in \cite[Theorem 8.1]{HHLN2} - see also \cite{yoneda} for a more elementary proof and a discussion of the subtlety behind this statement. 
\pend\end{rmk}
\begin{obs}
$\LFib^{rep}\subset \Fun(\Delta^1,\Cat)$ is closed under (cartesian) products, and the functor $ev_1$ preserves those. It is therefore canonically symmetric monoidal for the cartesian symmetric monoidal structure. 

Furthermore, $ev_1$-coCartesian edges are preserved under cartesian products, as is easily seen from their description above. 
\pend\end{obs}
We note that the ofllowing holds for general reasons:
\begin{lm}\label{lm:lfibcart}
    The symmetric monoidal structure on the functor $\Cat\to\Cat, \C\mapsto \C\op$ classified by the symmetric monoidally coCartesian fibration $ev_1: \LFib^{rep}\to\Cat$ is strict monoidal, and is therefore given by the canonical enhancement of product-preserving functors to symmetric monoidal functors. 
\end{lm}
\begin{proof}
    See the discussion in \cite[Corollaries 2.4.1.8 and 2.4.1.9]{HA}. 
\end{proof}
We want to provide a similar analysis for $\coCart\subset \Fun(\Delta^1,\Cat)$, the subcategory spanned by coCartesian fibrations and functors that send coCartesian edges to coCartesian edges.
We begin with a proposition which is analogous to \cite[Proposition 6.2]{haugchu}: 
\begin{prop}
The functor $ev_1: \coCart\to \Cat$ is a coCartesian fibration. 
\end{prop}
\begin{proof}
As pullbacks of coCartesian fibrations are coCartesian fibrations, it is clear that $\coCart\subset \Fun(\Delta^1,\Cat)$ is a sub-\emph{cartesian} fibration. 

In particular, it is a cartesian fibration. To prove that it is also coCartesian, it therefore suffices to prove that the pullback functors have left adjoints (\cite[Corollary 5.2.2.5]{HTT}).

If $f: \C\to \D$ is a functor, then the pullback functor $\coCart_\D\to \coCart_\C$ gets identified, under the straightening/unstraightening equivalence, with $f^* :~\Fun(\D,\Cat)\to~\Fun(\C,\Cat)$, which has a left adjoint given by left Kan extension. This proves the claim. 
\end{proof}
\begin{rmk}\label{rmk:notcocart}
    We warn the reader that while $\coCart\subset \Fun(\Delta^1,\Cat) $ is an inclusion of cartesian fibrations which are both cocartesian, it is \emph{not} a morphism of coCartesian fibrations, i.e. it does not preserve coCartesian edges. To see this, it suffices to observe that coCartesian edges in $\Fun(\Delta^1,\Cat)$ are simply the ones that induce equivalences on sources, so that the induced functors send $E\to \C$ to the composite $E\to \C\to\D$. It thus suffices to find a composite of a coCartesian fibration with an arbitrary functor which is no longer coCartesian. Examples of this abound, such as $\C^{\Delta^1}\xrightarrow{ev_1}\C$ followed by $\C\xrightarrow{\Delta}\C^{\Delta^1}$\footnote{More generally, suppose $E\to \C$ is a coCartesian fibration with no empty fibers, and suppose $\C\to\D$ is a functor such that there are objects in $\D$ not in the image of $\C$, that do receive maps from objects in the image. The composite $E\to\D$ is not coCartesian, since the induced functors would have to go from nonempty fibers to empty fibers.}. 
\pend\end{rmk}

We note that $\coCart\subset \Fun(\Delta^1,\Cat)$ is obviously closed under products, and those are preserved under $ev_1$. To conclude our analysis we now prove the following, which is analogous to \cite[Proposition 6.10]{haugchu}: 
\begin{prop}\label{prop : cocartcart}
In the coCartesian fibration $\coCart\xrightarrow{ev_1} \Cat$, coCartesian edges are preserved under products.
\end{prop}
\begin{proof}
This can be formulated as saying that for any pair of functors $\alpha: \C_0\to \C_1, \beta:~\D_0\to~\D_1$, the following natural transformation is an equivalence: 

\[\begin{tikzcd}
	{\coCart_{\C_0}\times \coCart_{\D_0}} & {\coCart_{\C_1}\times\coCart_{\D_1}} \\
	{\coCart_{\C_0\times \D_0}} & {\coCart_{\C_1\times\D_1}}
	\arrow[from=1-1, to=2-1]
	\arrow[from=1-2, to=2-2]
	\arrow["(\alpha\times\beta)_!"',from=2-1, to=2-2]
	\arrow["\alpha_!\times\beta_!",from=1-1, to=1-2]
	\arrow[shorten <=10pt, shorten >=10pt, Rightarrow, from=2-1, to=1-2]
\end{tikzcd}\]

i.e. the square strictly commutes, instead of only commuting up to a non-invertible natural transformation (this natural transformation comes from the fact that $$\coCart\times\coCart\xrightarrow{\times}\coCart\times_{\Cat}(\Cat\times\Cat)$$ is a morphism over $\Cat\times\Cat$, between cocartesian fibrations; and so we get canonical lax commutative squares as above). 

Under the straightening/unstraightening equivalence, the functor $\coCart_{\C_0}\times~\coCart_{\D_0}\to~\coCart_{\C_0\times\D_0}$ gets identified with $\Fun(\C_0,\Cat)\times\Fun(\D_0,\Cat)\to \Fun(\C_0\times\D_0,\Cat)$ given by $(F,G)\mapsto~((c,d)\mapsto~F(c)\times~G(d))$, or more precisely, $\Pi\circ p_{\C_0}^*F \times p_{\D_0}^*G$, where $p_{\C_0}$ (resp. $p_{\D_0}$) is the projection from the product $\C_0\times\D_0\to \C_0$ (resp. $\D_0$) and $\Pi:\Cat\times\Cat\to \Cat$ is the product functor. 

Indeed, if $X\to\C_0, Y\to \D_0$ are coCartesian fibrations, there is an equivalence of coCartesian fibrations over $\C_0\times\D_0$ (natural in $X,Y$) $X\times Y\simeq p_{\C_0}^*X\times_{\C_0\times\D_0} p_{\D_0}^*Y$, this fiber product is the cartesian product in $\coCart_{\C_0\times\D_0}$, and finally restriction in functor categories corresponds to pullback in fibrations. 

We denote this abusively by $F\times G$, hoping that no confusion arises. 

One checks similarly that the canonical natural transformation from above, in terms of functors, and evaluated at $(F,G)$ is the canonical morphism $(\alpha\times\beta)_!(F\times G)\to \alpha_! F\times \beta_! G$. 

But this is an equivalence, because left Kan extensions are pointwise ($\Cat$ is cocomplete) so that the canonical morphism from above is pointwise of the form:  $$(\colim_{ (\C_0\times\D_0)\times_{\C_1\times\D_1}(\C_1\times\D_1)_{/(c,d)}} F\times G) \to (\colim_{ (\C_0\times_{\C_1}(\C_1)_{/c} )} F)\times (\colim_{ (\D_0\times_{\D_1}(\D_1)_{/d}  )} G)$$
and the comma categories appearing in the indexing of these colimits are equivalent: the canonical functor  $$(\C_0\times\D_0)\times_{\C_1\times\D_1}(\C_1\times\D_1)_{/(c,d)}\to (\C_0\times_{\C_1}(\C_1)_{/c} )\times (\D_0\times_{\D_1}(\D_1)_{/d}  ) $$ is an equivalence - here, we use that for any $C\in\Cat$, $C\times -$ commutes with colimits to reduce to a comparison between indexing categories. 
This concludes the proof.
\end{proof}
\begin{rmk}
It is likely that one could also give a proof without using the straightening/unstraightening equivalence, but using a description of left Kan extensions at the level of coCartesian fibrations. Such a description is not hard to guess using the description of left Kan extensions in terms of comma categories (and the description of colimits in terms of localizations) and one can most likely prove it using \cite{gepnerhaugsengnik}. Namely, the guess is that if $\alpha: \C_0\to C_1$ is a functor and $f:\E_0\to \C_0$ a coCartesian fibration classifying $F:\C_0\to \Cat$, then the coCartesian fibration classifying the left Kan extension $\alpha_! F$ is classified by $(\E_0\times_{\C_1}\Fun(\Delta^1,\C_1))[W^{-1}]$ where $W$ is the class of maps whose projection to $\E_0$ is $f$-coCartesian (this is supposed to represent the pointwise formula for left Kan extensions, using the way colimits in $\Cat$ are computed). 

The description that one gets this way is clearly stable under products. 
\pend\end{rmk}
We use \Cref{prop : cocartcart} to define a lax symmetric monoidal structure on $\C\mapsto \coCart_\C$:
\begin{defn}\label{defn : cancocart}
Let $\C\mapsto \coCart_\C$ denote the functor $\Cat\to \widehat{\Cat}$ classified by $\coCart\to~\Cat$.\footnote{One can compare this functoriality to the one used in \cite[Appendix A]{gepnerhaugsengnik}, by observing that our contravariant functoriality is a subfunctor of $\Cat_{/\C}$, just like the one from that paper. In other words, the functorality described just here is obtained from the one in that paper by taking left adjoints. We will not need this fact.  }

\Cref{prop : cocartcart} shows that $\coCart^\times \to \Cat^\times$ is a symmetric monoidally coCartesian fibration, so we may apply \Cref{thm : microcosm} to get a lax symmetric monoidal structure on the functor it classifies. We call this the canonical lax symmetric monoidal structure on the functor $\C\mapsto\coCart_\C$.  
\pend\end{defn}
In other words, what we proved is that the situation for $\coCart\to \Cat$ is the same as for $\LFib^{rep}\to \Cat$. We furthermore note that $\LFib^{rep}$ is included in $\coCart$ (both lie inside $\Fun(\Delta^1,\Cat)$), so that the following makes sense :
\begin{lm}
The inclusion $\LFib^{rep}\to \coCart$ preserves $ev_1$-coCartesian morphisms, and products. 
\end{lm}
\begin{proof}
The claim about products is obvious.

For $ev_1$-coCartesian morphisms, we again use straightening: by looking at fibers, the claim is that for any $f: \C\to \D$, the natural transformation in the following square is an equivalence: 
\[\begin{tikzcd}
	{\LFib^{rep}_\C} & {\coCart_\C} \\
	{\LFib^{rep}_\D} & {\coCart_\D}
	\arrow[from=1-1, to=1-2]
	\arrow[from=1-1, to=2-1]
	\arrow[from=2-1, to=2-2]
	\arrow[from=1-2, to=2-2]
	\arrow[shorten <=5pt, shorten >=5pt, Rightarrow, from=1-2, to=2-1]
\end{tikzcd}\]
Under straightening, this square gets identified with :
\[\begin{tikzcd}
	\C\op & {\Fun(\C,\Cat)} \\
	\D\op & {\Fun(\D,\Cat)}
	\arrow[from=1-1, to=2-1]
	\arrow[from=1-1, to=1-2]
	\arrow[from=2-1, to=2-2]
	\arrow[from=1-2, to=2-2]
	\arrow[shorten <=6pt, shorten >=6pt, Rightarrow, from=1-2, to=2-1]
\end{tikzcd}\]
and the canonical transformation being invertible is simply a witness of the naturality of the Yoneda embedding (and the fact that the inclusion $\Ss\subset \Cat$ preserves colimits). 
\end{proof}

\begin{cor}\label{cor : yonmon}
The inclusion $\LFib^{rep}\to \coCart$ induces a symmetric monoidal natural transformation of lax symmetric monoidal functors $\Cat\to \widehat{\Cat}$: $$\C\op\to \coCart_\C$$
\end{cor}
\begin{proof}
Both $\LFib^{rep}$ and $\coCart$ have finite products, and the functor $ev_1$ preserves products for both, so we can see it as a symmetric monoidal functor to $\Cat$, where both the domain and codomain are viewed as cartesian symmetric monoidal categories. 

Furthermore, we have seen that coCartesian edges are closed under product, so that they are both symmetric monoidally coCartesian fibrations over $\Cat$. 

\Cref{thm : microcosm} implies (up to passing to a larger universe) that their straightening is canonically lax symmetric monoidal, and the transformation between them too. 
\end{proof}
\begin{prop}\label{prop : liftcocart}
The lax symmetric monoidal functor $\Cat\to \widehat{\Cat}, \C\mapsto \coCart_\C$ classified by $\coCart$ has a canonical lax symmetric monoidal lift to $\Fun(\Cat, \Mod_{\Cat}(\PrL))$ which, pointwise, agrees with the $\Cat$-linear structure induced from the inclusion $\coCart_\C\subset~\Cat_{/\C}$. 
\end{prop}
By the latter, we mean that $\Cat_{/\C}$ is naturally a $\Cat$-module, and $\coCart_\C\subset\Cat_{/\C}$ is stable under this structure, and therefore inherits a $\Cat$-module structure as well. 
\begin{rmk}
Informally, this module structure is easy to describe:  given $\D\to \C\in~\coCart_\C$ and $K\in \Cat$, take the action of $K$ on $\D$ to be $\D\times K\to \D\to \C$, where $\D\times K\to\D$ is the projection map. 
\pend\end{rmk}
The point of this proposition is essentially that $*$ is the unit of $\Cat$, and this is sent to $\Cat$ in $\widehat{\Cat}$ (moving from $\widehat{\Cat}$ to $\PrL$ is something that can be \emph{checked}, and we have nothing to construct). To make this precise, we have the following lemma, which is almost tautological, but which we will also use later on: 

\newcommand{\F}{\cat F }
\begin{lm}\label{lm A}
Let $\E\to \F$ be an $\Oo$-monoidally coCartesian fibration, $\Oo'$ an operad over $\Oo$.

Let $f\in\Alg_{\Oo'/\Oo}(\F)$. The following two constructions equip the fiber $\E_f$ with an  $\Oo'$-monoidal structure: 
\begin{enumerate}
    \item The functor classified by $\E$ is canonically lax $\Oo$-monoidal $\F\to\Cat$ via \Cref{thm : microcosm} and therefore sends the $\Oo'$-algebra $f$ to an $\Oo'$-algebra $\E_f$ in $\Cat$, which is the same thing as an $\Oo'$-monoidal category .
    \item One can take the following pullback as in \Cref{prop : pullback}: $$\xymatrix{\E_f^\otimes \ar[r] \ar[d] & \E^\otimes \ar[d]\\ (\Oo')^\otimes \ar[r]^f & \F^\otimes}$$
\end{enumerate}
These two $\Oo'$-monoidal structures agree, and $\Alg_{\Oo'/\Oo}(\E_f)$ is the fiber of $\Alg_{\Oo'/\Oo}(\E)\to~\Alg_{\Oo'/\Oo}(\F)$ over $f$. 
\end{lm}

\begin{proof}[Proof of \Cref{lm A}]
The second claim about the fiber is immediate from the second construction, so we focus on the first part of the lemma. 

The observation here is that the lax $\Oo$-monoidal functor $\F^\otimes\to \Cat^\times$ associated to $\E^\otimes\to \F^\otimes$ fits in a pullback square $$\xymatrix{\E^\otimes \ar[r] \ar[d] & (\Cat_{*\sslash})^\times\ar[d]\\ \F^\otimes\ar[r] & \Cat^\times}$$
so in particular, we have a bigger pullback diagram 
$$\xymatrix{\E_f^\otimes\ar[r] \ar[d] & \E^\otimes \ar[r] \ar[d] & (\Cat_{*\sslash})^\times\ar[d]\\ (\Oo')^\otimes \ar[r] & \F^\otimes\ar[r] & \Cat^\times}$$

The top left hand corner is the $\E_f^\otimes$ from the first construction, while the bottom horizontal composite by definition classifies $\E_f$ as an $\Oo'$-algebra in $\Cat$ via the lax symmetric monoidal functor $\F^\otimes\to \Cat^\times$. 

The equivalence between $\Oo'$-algebras in $\Cat$ and $\Oo'$-monoidal categories now precisely tells us that $\E_f^\otimes$ is the $\Oo'$-monoidal category corresponding to the $\Oo'$-algebra $\E_f$ in $\Cat$. 
\end{proof}

\begin{proof}[Proof of \Cref{prop : liftcocart}]
We begin by proving the statement replacing $\PrL$ with $\widehat{\Cat}$.

For $\widehat{\Cat}$, we observe that $\C\mapsto \coCart_\C$, as a functor $\Cat\to \widehat{\Cat}$, is lax symmetric monoidal, and therefore induces a canonical lax symmetric monoidal functor $\Cat\to~\Mod_{\coCart_*}(\widehat{\Cat})$, because $*$ is the unit in $\Cat$. 

We are therefore left with two verifications: firstly, that the equivalence $\coCart_* \simeq \Cat$ can be made symmetric monoidal, i.e. that $\coCart_*$ is cartesian monoidal; and secondly, the claim about the pointwise $\Cat$-linear structure on $\coCart_\C$. 

For the first part, we apply \Cref{lm A}  to $\Oo=\Oo'= \mathrm{Comm}, \E^\otimes= \coCart^\times, \F= \Cat^\times$, from which it follows that the monoidal structure on $\coCart_*$ is cartesian, and hence, the correct one.

For the second claim, we again use \Cref{lm A}, with $\Oo= \mathrm{Comm}$ and $\Oo'=$ the operad classifying left modules, together with the observation that we have a map $\coCart^\times\to~\Fun(\Delta^1,\Cat)^\times$ over $\Cat^\times$. We then only have to \emph{check} that the map we obtain on fibers over $(\Oo')^\otimes\xrightarrow{(*,\C)}~\Cat^\times$, which is a priori a map of $\Oo'$-operads, is a map of $\Oo'$-monoidal categories - but this follows from the fact that $\coCart_\C \subset \Cat_{/\C}$ is stable under $-\times K$ for $K\in\Cat$.

We now explain how to replace $\widehat{\Cat}$ with $\PrL$: the Lurie tensor product by definition witnesses $(\PrL)^\otimes$ as a (non-full) sub-operad of $\widehat{\Cat}^\times$, so moving from the latter to the former is just a \emph{property}, of the objects involved (namely that of being presentable) and the functors involved (preserving colimits in each variable). In our situation, the objects are $\coCart_\C$ for some $\C \in \Cat$, and these are indeed presentable, and the morphisms are either $\Cat\times\coCart_\C\to \coCart_\C$ (for the $\Cat$-module structure) or $\coCart_\C\times~\coCart_\D \to~\coCart_{\C\times \D}$ (for the lax symmetric monoidal structure). To verify that these do preserve colimits in each variable, one can either use un/straightening, or the results from \Cref{section : straight}. Let us indicate how the latter verification works: \Cref{prop : colim} allows us to reduce to the same statements for $\Cat_{/\C}, \Cat_{/\D}$. For these, one may use that colimits in slice categories are underlying, so we are reduced to the same statement for $\Cat$, but now for any $E\in\Cat, E\times -$ has a right adjoint and hence preserves colimits. 

All in all, we get a lift to $\Mod_\Cat(\PrL)$. 
\end{proof}
In order to state and prove \Cref{thm  : metacosm}, we have to define a symmetric monoidal structure on $\C\mapsto \Fun(\C,\Cat)$. Firstly, by \cite[Proposition 4.8.1.3]{HA} (see \cite[Remark 4.8.1.8]{HTT}, applied to $\mathcal K = \emptyset, \mathcal K'= $ all simplicial sets), there is a canonical symmetric monoidal structure on the functor $\Cat\to \PrL, \C\mapsto \Fun(\C\op,\Ss)$. 

Furthermore, $\C\mapsto\C\op$ is product preserving and so uniquely symmetric monoidal as a functor $\Cat\to \Cat$ (see also \Cref{lm:lfibcart}). The functor $-\otimes \Cat: \PrL\to \Mod_\Cat(\PrL)$ also has a canonical symmetric monoidal structure. 
\begin{defn}\label{defn : canday}
We define the canonical symmetric monoidal structure on $$\Cat\to \Mod_\Cat(\PrL), \C\mapsto \Fun(\C,\Cat)$$ as the composite  $$\Cat\xrightarrow{\C\mapsto \C\op} \Cat\xrightarrow{\D\mapsto \Fun(\D\op,\Ss)}\PrL\xrightarrow{-\otimes\Cat}\Mod_\Cat(\PrL)$$
followed by the natural equivalence $\Fun(\C,\Cat)\simeq\Fun((\C\op)\op,\Ss)\otimes\Cat$. We call it the canonical symmetric monoidal
structure on $\C\mapsto \Fun(\C,\Cat)$. 
\pend\end{defn}
Here the symmetric monoidal functor $\Cat\to \PrL, \D\mapsto \Fun(\D\op,\Ss) $ is the one from \cite[Remark 4.8.1.8]{HA} with (in the notation from \textit{loc. cit.}) $\mathcal K=\empty,\mathcal K'=$ all diagrams. This is the symmetric monoidal left adjoint to the functor described in \cite[Corollary 4.8.1.4]{HA}. Note that, forgetting the symmetric monoidal structure, the functoriality here is the one obtained from the universal property of the presheaf category which, by naturality of the Yoneda embedding \cite[Theorem 8.1]{HHLN2}\footnote{As in \Cref{rmk : naturalyoneda}, see also \cite{yoneda} for a discussion.}, is the same as the functoriality obtained by taking left adjoints on the natural contravariant functoriality. 

Finally, the equivalence $\Fun(\C,\Ss)\otimes\Cat\simeq \Fun(\C,\Cat)$ is the one from, e.g. \cite[Theorem 4.2(2)]{HR} which we sketch here for convenience of the reader: as the target is naturally a $\Cat$-module, the map $\Fun(\C,\Ss)\to \Fun(\C,\Cat)$ induces a map $\Fun(\C,\Ss)\otimes~\Cat\to~\Fun(\C,\Cat)$ which can be checked to be an equivalence pointwise, using \cite[Proposition 4.8.1.17]{HA}. 

\begin{rmk}\label{rmk:content}
    We make a remark about the choice of symmetric monoidal structure on the functor $\C\mapsto \Fun(\C,\Cat)$. An alternative is to consider the lax slice $\widehat{\Cat}_{\sslash \Cat} \to \widehat{\Cat}$ as a symmetric monoidal coCartesian fibration using the cartesian monoidal structure on both source and target. This fibration, as a \emph{cartesian} fibration, classifies the contravariant functoriality of $\C\mapsto \Fun(\C,\Cat)$ together with its ``obvious'' lax symmetric monoidal structure. 

    As a coCartesian fibration, as mentioned above, it classifies the same functoriality as the ``free ($\Cat$-linear) cocompletion'' functoriality, and for this choice of a symmetric monoidal structure, \Cref{thm : metacosm} would be essentially obvious, since $\widehat{\Cat}_{\sslash \Cat}$ is equivalent to $\coCart$ and their cartesian symmetric monoidal structures therefore are as well. The content\footnote{In some sense, the main ``non-formal'' point of the present paper} of \Cref{thm : metacosm}, as stated precisely below, is therefore that the monoidal structure on $\widehat{\Cat}_{\sslash \Cat}$ we just described agrees with the one arising from \Cref{defn : canday}. One can see this as a symmetric monoidal refinement of the previous discussion, or equivalently, of \cite[Theorem 8.1]{HHLN2}.
\pend\end{rmk}

\begin{obs}
The forgetful functor $\Mod_\Cat(\PrL)\to \PrL\to \widehat{\Cat}$ is canonically lax symmetric monoidal, so we can also view $\C\mapsto \Fun(\C,\Cat)$ as a lax symmetric monoidal functor $\Cat\to\widehat{\Cat}$. 
\pend\end{obs}
We can now prove \Cref{thm : metacosm} in the following form: 
\begin{cor}
There is a $\Cat$-linearly symmetric monoidal natural transformation of lax symmetric monoidal functors $\Cat\to \Mod_{\Cat}(\PrL)$ of the form $$\Fun(\C,\Cat)\to \coCart_\C$$
such that for any category $\C$, the corresponding functor is the unstraightening equivalence. 

In particular, these two functors are $\Cat$-linearly symmetric monoidally equivalent. 
\end{cor}
\begin{proof}
This follows essentially from \Cref{cor : yonmon}, the universal property of presheaf-categories, and the fact that $\Fun(\C,\Ss)\otimes\Cat \simeq \Fun(\C,\Cat)$.

In more detail, we have just explained how $\C\mapsto \coCart_\C$ lifts to a symmetric monoidal functor $\Cat\to \Mod_\Cat(\PrL)$, and the forgetful functor $\PrL\to \widehat{\Cat}$ has a partial symmetric monoidal left adjoint defined on small categories, and given there by $\Fun((-)\op,\Ss)$. 

$\Mod_\Cat(\PrL)\to \PrL$ has a symmetric monoidal left adjoint given by $\Cat\otimes -$, and $\Fun((-)\op,\Ss)\otimes\Cat\simeq \Fun((-)\op,\Cat)$, symmetric monoidally by definition. 

It therefore follows that the symmetric monoidal transformation $\C\op\to \coCart_\C$ from \Cref{cor : yonmon}  has a unique $\Cat$-linear colimit-preserving extension $\Fun(\C,\Cat)\to \coCart_\C$, which has a canonical $\Cat$-linear symmetric monoidal structure. 

It is now easy to see that the underlying functor $\Fun(\C,\Cat)\to \coCart_\C$ is just the unstraightening equivalence : it is a colimit-preserving $\Cat$-linear functor $\Fun(\C,\Cat)\to~\coCart_\C$ whose restriction along the Yoneda embedding $\C\op\to \Fun(\C,\Ss)\to \Fun(\C,\Cat)$ agrees with the restriction of the unstraightening equivalence.
The unstraightening equivalence also has these properties (it is an equivalence, so it preserves colimits, and it is $\Cat$-linear - this is folklore, but see \Cref{lm : catlinear} for a proof), and this implies that they are equivalent by the universal property of presheaves and the equivalence $\Fun(\C,\Cat)\simeq \Fun(\C,\Ss)\otimes\Cat$. 

Because equivalences of symmetric monoidal functors are pointwise, i.e. the evaluation functors at $\C, \C\in\Cat$, are jointly conservative as functors $\Fun^\otimes(\Cat, \Mod_\Cat(\PrL))\to~\widehat{\Cat}$, the rest of the claim follows. 
\end{proof}
\begin{rmk}
A variation on the fact that $\Cat$ has a discrete anima of automorphisms shows that there can be only one natural un/straightening equivalence, see e.g. \cite[Appendix A]{HHLN2}. In particular, the un/straightening equivalence that we obtain here is equivalent to the one from \cite[Appendix A]{gepnerhaugsengnik}.  Note that our construction is, however, not independent of theirs and relies on work which in turn relies on \cite{gepnerhaugsengnik} - however, what we mean here is that the \emph{a priori} different naturality structure we obtain is in fact the same. 
\pend\end{rmk}

\section{The macrocosmic monoidal Grothendieck construction}\label{section : macro}
If $\C$ is symmetric monoidal (or more generally $\Oo$-monoidal), the symmetric monoidal structures on the two functors $$\D\mapsto \Fun(\D,\Cat), \D\mapsto \coCart_\D$$ induce symmetric monoidal structures on $\Fun(\C,\Cat)$ and $\coCart_\C$ respectively, and the symmetric monoidal structure on the natural equivalence between the two makes the unstraightening equivalence $$\Fun(\C,\Cat)\simeq\coCart_\C$$ symmetric monoidal. 

An important question is a description of these symmetric monoidal structures. This is the content of this section.

We begin with the easier side of our analysis (see \cite[Remark 4.8.1.13]{HA}): 
\begin{lm}\label{lm : abstractisday}
Let $\C$ be an $\Oo$-monoidal category. The $\Oo$-monoidal structure induced on $\Fun(\C,\Cat)$ is the $\Oo$-monoidal Day convolution structure.
\end{lm}
\begin{proof}
Because both this structure and the Day convolution structure are presentably $\Cat$-linearly $\Oo$-monoidal, it suffices to prove that the Yoneda embedding $\C\op\to \Fun(\C,\Cat)$ is $\Oo$-monoidal, where the target has Day convolution and the source has the monoidal structure coming from the symmetric monoidality of $\Cat\to \Cat, \D\mapsto \D\op$. 

But this symmetric monoidality comes from the fact that this functor preserves products, cf. \Cref{lm:lfibcart}. In particular, the structure on $\C\op$ is just the usual $\op$ of $\Oo$-monoidal structures, and the claim boils down to the claim that for an $\Oo$-monoidal category, the Yoneda embedding $\C\to \Fun(\C\op,\Ss)$ is $\Oo$-monoidal. 

This is \cite[Proposition 4.8.1.10]{HA} - more specifically the variant of \cite[Proposition 4.8.1.12]{HA}, but with a more general operad than $\mathrm{Comm}$. 
\end{proof}
We now move on to the $\coCart_\C$ part of the picture. Suppose $\C$ is an $\Oo$-monoidal category - then $\Cat_{/\C}$ is canonically $\Oo$-monoidal, and our goal will be to identify $(\coCart_\C)^\otimes \to~(\Cat_\C)^\otimes$ as an \emph{explicit} non-full sub-$\Oo$-operad inclusion. 

The key lemma for this will be \Cref{lm A} from the previous section -  what we prove here is a refinement of our earlier proof that the monoidal structure of $\coCart_*$ is the cartesian monoidal one. We now state and prove the main result of this section : 

\begin{lm}\label{thm : macroprecise}
Let $\C$ be an $\Oo$-monoidal category for some operad $\Oo$. The $\Oo$-monoidal structure on $\coCart_\C$ induced by \Cref{defn : cancocart} witnesses it as a non-full sub-$\Oo$-operad of $(\Cat_{/\C})^\otimes$ where we can describe the multi-mapping anima as follows: 

Given $T_1, ..., T_n, T_\infty \in \Oo^\otimes_{\langle 1\rangle}$ and an operation $e: T_1\oplus ...\oplus T_n\to T_\infty$, letting $e_! :~\C_{T_1}\times~... \times~\C_{T_n}\to~\C_{T_\infty}$ denote the corresponding operation, if $\D_i\to \C_{T_i}$ is a coCartesian fibration for each $i\in~\{1,...,n,\infty\}$, then $$\map_{\coCart_\C^\otimes}(\D_1,...,\D_n; \D_\infty)_e \to \map_{(\Cat_{/\C})^\otimes}(\D_1,...,\D_n; \D_\infty)_e \simeq \map_{\Cat_{/\prod_{i=1}^n\C_{T_i}}}(\prod_{i=1}^n\D_i, \prod_{i=1}^n\C_{T_i}\times_{\C_{T_\infty}}\D_\infty) $$ 
witnesses the former as the subanima of the latter spanned by the components corresponding to functors of coCartesian fibrations over $\prod_{i=1}^n\C_{T_i}$. In other words, functors $\prod_{i=1}^n\D_i\to~\D_\infty$ lying over $e_! : \prod_{i=1}^n\C_{T_i}\to \C_{T_\infty}$ which preserve coCartesian morphisms in each variable. 
\end{lm}
It is worth spelling out the case of $\Oo=\mathrm{Comm}$ for intuition: 
\begin{cor}
Let $\C$ be a symmetric monoidal category. The  symmetric monoidal structure on $\coCart_\C$ induced by \Cref{defn : cancocart} witnesses it as a non-full sub-operad of $(\Cat_{/\C})^\otimes$, where we can describe the multi-mapping anima as follows: 

Given $n\geq 0$ and letting $e:\langle n\rangle \to \langle 1\rangle$ denote the unique active morphism and $\otimes :~\C^n\to~\C$ denote the corresponding tensor product, if $\D_i\to \C$ is a coCartesian fibration for $i\in\{1,...,n,\infty\}$, then 
$$\map_{\coCart_\C^\otimes}(\D_1,...,\D_n; \D_\infty)_e \to \map_{(\Cat_{/\C})^\otimes}(\D_1,...,\D_n; \D_\infty)_e \simeq \map_{\Cat_{/\C^n}}(\prod_{i=1}^n\D_i, \C^n\times_\C\D_\infty) $$ 
witnesses the former as the subanima of the latter spanned by the components corresponding to functors of coCartesian fibrations over $\C^n$. 

In other words, functors $\prod_{i=1}^n\D_i\to \D_\infty$ lying over $\otimes : \C^n\to \C$ which preserve coCartesian morphisms in each variable. 
\end{cor}
\begin{proof}[Proof of \Cref{thm : macroprecise}]
By definition, \Cref{thm : metacosm} induces an $\Oo$-monoidal structure on $\coCart_\C$ in the following way : $\C$ is an $\Oo$-algebra in $\Cat$, and $\coCart_\bullet : \Cat\to \widehat{\Cat}$ is obtained from \Cref{thm : microcosm} by straightening the symmetric monoidally coCartesian fibration $\coCart\to \Cat$. 

Up to passing to a bigger universe, we are therefore exactly in the situation of \Cref{lm A}. In particular, we have a pullback square of the form $$\xymatrix{\coCart_\C^\otimes \ar[r]\ar[d] & \coCart^\times \ar[d]\\ \Oo^\otimes \ar[r] & \Cat^\times}$$
Note that both $\coCart,\Cat$ are cartesian monoidal. 

We similarly have, by applying the same argument to $\Fun(\Delta^1,\Cat)$  (which is easily seen to be cartesian monoidal) in place of $\coCart$, a pullback square: 
$$\xymatrix{(\Cat_{/\C})^\otimes \ar[r]\ar[d] & \Fun(\Delta^1,\Cat)^\times \ar[d]\\ \Oo^\otimes \ar[r] & \Cat^\times}$$
Recall from \Cref{rmk:notcocart} that the map $\coCart\to \Fun(\Delta^1,\Cat)$ is not a morphism of coCartesian fibrations, but it is nonetheless a morphism over $\Cat$ (in fact it is a morphism of \emph{cartesian} fibrations over $\Cat$) so that we also have a commutative triangle : 

$$\xymatrix{\coCart^\times \ar[r] \ar[d] & \Fun(\Delta^1,\Cat)^\times \ar[dl]\\
\Cat^\times}$$

Because $\coCart\subset \Fun(\Delta^1,\Cat)$ preserves products and is a non-full subcategory, $\coCart^\times\subset \Fun(\Delta^1,\Cat)^\times$ is a non-full suboperad. 

Explicitly, given coCartesian fibrations $\D_1\to\C_1, ..., \D_n\to \C_n, \D_\infty\to \C_\infty$, which we will abusively denote by only referring to their total category,  the inclusion $$\map_{\coCart^\times}(\D_1,...,\D_n; \D_\infty)\subset \map_{\Fun(\Delta^1,\Cat)^\times}(\D_1,...,\D_n; \D_\infty)$$ is the inclusion of components corresponding to morphisms $\D_1\times ... \times\D_n\to \D_\infty$ lying over a morphism $\C_1\times ... \times\C_n\to \C_\infty$ and sending (factor-wise) coCartesian edges to coCartesian edges.  

The two pullback squares above and the commutative triangle induce the following pullback square $$\xymatrix{\coCart_\C^\otimes \ar[r]\ar[d] & (\Cat_{/\C})^\otimes \ar[d]\\ \coCart^\times \ar[r] & \Cat^\times}$$

The description of $\coCart^\times \subset \Cat^\times$ above clearly specializes to the claimed description of $\coCart_\C^\otimes\subset (\Cat_{/\C})^\otimes$. 
\end{proof}

We summarize the work of this section in : 
\begin{cor}\label{cor : macromongro}
Let $\Oo$ be an operad and $\C$ an $\Oo$-monoidal category. 

The non-full sub-$\Oo$-operad $\coCart_\C^\otimes \subset (\Cat_{/\C})^\otimes$ is an $\Oo$-monoidal category, and this $\Oo$-monoidal structure is the one inherited from the symmetric monoidality of the functor $\D\mapsto \coCart_\D$ from \Cref{defn : cancocart}. 

In particular, \Cref{thm : metacosm} induces an $\Oo$-monoidal equivalence $$\Fun(\C,\Cat)\simeq \coCart_\C$$ between the Day convolution $\Oo$-monoidal structure and this specific $\Oo$-monoidal structure, whose underlying functor is the unstraightening equivalence. 
\end{cor}

We also deduce the following description of algebras in $\coCart_\C$ :
\begin{prop}\label{prop : algincocart}
let $\Oo$ be an operad and $\C$ an $\Oo$-monoidal category. The inclusion $\coCart_\C\to\Cat_{/\C}$ induces an equivalence of categories between $\Oo$-algebras in $\coCart_\C$ and $\Oo$-monoidally coCartesian fibrations over $\C$ (see \Cref{defn : monfib}): $$\Alg_\Oo(\coCart_\C)\simeq \coCart_\C^\Oo$$
\end{prop}
\begin{proof}
$\coCart_\C \subset \Cat_{/\C}$ is a non-full sub-$\Oo$-operad, so in particular $\Alg_\Oo(\coCart_\C)\to~\Alg_\Oo(\Cat_{/\C})$ is a non-full subcategory. 

By \cite[Lemma 2.12]{thom}, $\Alg_\Oo(\Cat_{/\C}) \simeq \Alg_\Oo(\Cat)_{/\C}$ and by the equivalence between $\Alg_\Oo(\Cat)$ and $\Oo$-monoidal categories, the latter is a non-full subcategory of $\Cat_{/\C^\otimes}$. 

This chain of functors witnesses $\Alg_\Oo(\coCart_\C)$ as a non-full subcategory of $\Cat_{/\C^\otimes}$. Examining at each stage the essential image yields the desired claim.
\end{proof}
\begin{cor}
Let $\Oo$ be an operad and $\C$ an $\Oo$-monoidal category. Taking $\Oo$-algebras on both sides of the equivalence from \Cref{cor : macromongro} yields an equivalence of categories: $$\Fun^{lax-\Oo}(\C,\Cat)\simeq \coCart^\Oo_\C$$ 
\end{cor}
\begin{proof}
This follows from \Cref{prop : algincocart} and the description of $\Oo$-algebras in the Day convolution monoidal structure. 
\end{proof}
\begin{rmk}
Of course, this last corollary does not require our main result, and can simply be seen as instance of un/straightening over $\C^\otimes$, once one knows that lax $\Oo$-monoidal functors $\C\to\Cat$ can be described as certain (ordinary) functors $\C^\otimes\to\Cat$ (and not $\Cat^\times$!).
\pend\end{rmk}

Passing to anima-valued (or $\infty$-groupoid-valued) functors instead of category valued functors corresponds to passing to left fibrations instead of coCartesian fibrations under the unstraightening equivalence. Letting $\Ss\subset \Cat$ denote the full subcategory spanned by anima, we automatically get \Cref{cor : corlfib}: 
\begin{cor}
The anima-valued un/straightening equivalence can be made into a symmetric monoidal equivalence of functors $\Cat\to \PrL$ : $$\Fun(\C,\Ss)\simeq \LFib_\C$$
If $\Oo$ is an operad and $\C$ an $\Oo$-monoidal category, this specializes to an $\Oo$-monoidal equivalence between the Day-convolution and some $\Oo$-monoidal structure on $\LFib_\C$ analogous to the one described in \Cref{thm : macroprecise}. Taking $\Oo$-algebras therein specializes to an equivalence $$\Fun^{lax-\Oo}(\C,\Ss)\simeq \LFib^\Oo_\C$$
where the latter is defined analogously to \Cref{defn : monfib}. 
\end{cor}
Specializing a bit more, and letting $\C$ be a anima $X$, in this case the inclusion $\LFib_X\subset \Ss_{/X}$ is an equivalence, we obtain the following folklore result (see e.g. \cite[Notation 3.1.2]{HLBrauer} in the case where $\Oo$ is the commutative operad) : 
\begin{cor}
Let $\Oo$ be an operad and $X$ an $\Oo$-algebra in anima. There is an equivalence of $\Oo$-monoidal categories $$\Fun(X,\Ss)\simeq \Ss_{/X}$$
where the left hand side has the Day convolution structure, while the right hand side has the comma category $\Oo$-monoidal structure. 
\end{cor}
\begin{rmk}
An independent proof of this result is much simpler than what we developped in this paper. Indeed, $\Ss$ has a universal property as a symmetric monoidal category, so to compare the lax symmetric monoidal functors $\Ss\to \PrL$, $X\mapsto \Ss_{/X}$ and $X\mapsto \Fun(X,\Ss)$, it suffices to check that they both preserve colimits and are in fact strict symmetric monoidal, rather than simply lax symmetric monoidal, and both claims are relatively easy. 
\pend\end{rmk}
\section{Comparison between macrocosmic and microcosmic versions}\label{section : comp}
We conclude this paper by outlining how a comparison between the macrocosmic version and the microcosmic version might go. 
\begin{warn}
In this section, some of the statements will not be completely proved, only conditional on some other statements. The conditional statements which we state nonetheless will be appended with a (*) symbol. 
\pend\end{warn}

Let us first explain what we wish to compare:  let $\C$ be an $\Oo$-monoidal category, and  $f: \C\to \Cat$ be a lax $\Oo$-monoidal functor. 

We have produced two ways to unstraighten it to an $\Oo$-monoidal fibration $\D\to \C$ :
\begin{enumerate}
    \item Using \Cref{thm : microcosm} directly; 
    \item We can view $f$ as an $\Oo$-algebra in $\Fun(\C,\Cat)$ (see \cite[Section 2.2.6]{HA}), use \Cref{thm  : macrocosm} to view it as an $\Oo$-algebra in $\coCart_\C$ and observe that these are exactly $\Oo$-monoidal fibrations over $\C$, by \Cref{prop : algincocart}.
\end{enumerate}
It is intuitively clear that these yield the same result : there should not be more than one way to turn monoidal fibrations into lax monoidal functors and conversely. 

In fact, it is not hard to convince oneself that this is true at a homotopically naive level. Namely, it is not hard to check that the informal description from the beginning of \Cref{section : pullback} (which is an accurate description of construction 1. above at the level of homotopy categories) is also a description of construction 2. above, unwinding all the definitions. 

What is more complicated is to do this coherently, and the author currently does not know how to. The goal of this section is to outline possible approaches to this question. For simplicity of exposition, we focus on the case $\Oo = \mathrm{Comm}$. 

The first possible approach, which we have not been able to carry through, is to try and prove a sort of rigidity statement, along the lines of: ``the space of self-equivalences of the functor $\C\mapsto \Fun^{lax-\otimes}(\C,\Cat)$ on $\CAlg(\Cat)$ is contractible'', following similar rigidity results for the category of categories itself, see e.g. \cite{BSP} and \cite[Appendix A]{HHLN2}.  

A second approach, which we could push just a bit further, involves $(\infty,2)$-categorical generalizations of our work, and we describe it now.

\begin{obs}\label{obs : suff}
Let $f:\C\to \Cat$ be as above. The microcosmic construction makes $\D^\otimes$ fit in a pullback $$\xymatrix{\D^\otimes \ar[r] \ar[d] & (\Cat_{*\sslash})^\times \ar[d]\\ \C^\otimes\ar[r] & \Cat^\times}$$ 

In particular, because $\D^\otimes$ and the abstract construction from the macrocosmic version agree at the level of underlying categories, it would suffice to construct a similar \emph{commutative square} with the latter as the top left hand corner, for then we would get a comparison map which would be an equivalence of underlying categories, and hence an equivalence. 

The issue in doing this with our current approach is that we want to use the naturality of our equivalence, but we only get naturality in \emph{strong} symmetric monoidal functors, because we only have control over commutative algebras in $\Cat$ - but $\C\to \Cat$ is only lax symmetric monoidal. 
\pend\end{obs}
Conditional on some simpler-looking $(\infty,2)$-categorical statement, we prove: 
\begin{cor}[*]\label{cor : comp}
In case $\C\to \Cat$ is strong symmetric monoidal, the two constructions agree.
\end{cor}
The simpler looking statement in question is the following: 
\begin{lm}[*]\label{lm : conditional}
There is a canonical commutative square $$\xymatrix{\CAlg(\Cat_{*\sslash})\ar[r]\ar[d] & \Cat_{*\sslash} \ar[d] \\ 
\CAlg(\Cat)\ar[r]^-{\CAlg(-)} & \Cat}$$
in which the vertical maps are forgetful maps. 
\end{lm}
The reason we view this as an $(\infty,2)$-categorical statement is that, because $\Cat$ and $\Cat_{*\sslash}$ are cartesian, $\CAlg(\Cat)$ (resp. $\CAlg(\Cat_{*\sslash})$) can be viewed as a certain full subcategory of $\Fun(\Fin_*,\Cat)$ (resp. $\Fun(\Fin_*,\Cat_{*\sslash}$), namely the one spanned by functors satisfying Segal conditions; and from this perspective, the functor $\CAlg(-): \CAlg(\Cat)\to \Cat$ can be viewed as the functor ``(partially) lax limit'' : $\Fun(\Fin_*,\Cat)\to \Cat$. 

This statement is then ``lax limits are well-defined functorially, and $\Cat_{*\sslash}\to \Cat$ preserves them'', which is intuitively clear. This statement should be simpler than the other $(\infty,2)$-categorical technology that we seem to need for the general statement, which is why we outline the proof of the special case mentioned above conditional to this statement in more detail. 

\begin{proof}[Proof of \Cref{cor : comp}]
Observe that $\CAlg(\widehat{\coCart}) \to \CAlg(\Fun(\Delta^1,\widehat{\Cat}))$ is a morphism of cartesian fibrations over $\CAlg(\widehat{\Cat})$ \cite[Theorem B.1]{HR}, and the former is a non-full subcategory of the latter. 

In particular, taking $f:\C\to \Cat$ as a morphism in the base, and taking a cartesian lift in $\CAlg(\widehat{\coCart})$ with target $\Cat_{*\sslash}$, we might as well take this cartesian lift in $\CAlg(\Fun(\Delta^1,\widehat{\Cat})) \simeq \Fun(\Delta^1,\CAlg(\widehat{\Cat}))$. 

This cartesian lift is then a morphism in this arrow category, hence a commutative square $$\xymatrix{\D\ar[r] \ar[d] & \Cat_{*\sslash} \ar[d] \\ \C\ar[r]_f & \Cat}$$ of symmetric monoidal categories and (strong !) symmetric monoidal functors. 

In particular, by \Cref{obs : suff}, it really suffices to show that this $\D$ that we obtain is the same as the one from using the construction following the macrocosm version, \Cref{thm  : macrocosm}.

For this, we observe that by \Cref{thm : metacosm} we have a commutative square of symmetric monoidal categories of the form $$\xymatrix{\Fun(\C,\widehat{\Cat}) \ar[r]^\simeq \ar[d]_{f_!} & \widehat{\coCart}_\C \ar[d]^{f_!}\\ \Fun(\Cat,\widehat{\Cat}) \ar[r]^\simeq  & \widehat{\coCart}_\Cat }$$ 
where the horizontal arrows are given by the transformation from \Cref{thm : metacosm}. 

Because the horizontal functors are equivalences, we can take right adjoints of the vertical maps and still have a commutative diagram, now of lax symmetric monoidal functors: $$\xymatrix{\Fun(\C,\widehat{\Cat}) \ar[r]^\simeq & \widehat{\coCart}_\C \\ \Fun(\Cat,\widehat{\Cat}) \ar[r]^\simeq \ar[u]^{f^*} & \widehat{\coCart}_\Cat \ar[u]_{f^*}}$$

Following along from $\C\mapsto \C$ in the bottom left hand corner, up-right gives $f$ and then the symmetric monoidal structure on $\D$ from construction 2. above. 

Going right-up gives $\Cat_{*\sslash}\to \Cat$ with its unique cartesian symmetric monoidal structure, and then the pullback of this. We are reduced to proving that this pullback coincides with the pullbacks from cartesian liftings from the beginning of the proof. By \Cref{lm A}, this follows from the next (conditional) lemma. 
\end{proof}
\begin{lm}[*]\label{lm B}
Let $\E\to \F$ be an $\Oo$-monoidally coCartesian fibration.
Assume that the underlying functor $\E\to\F$ is also a cartesian fibration. 

By \cite[Appendix B]{HR}, $\CAlg(\E)\to \CAlg(\F)$ is also a cartesian fibration, in particular, given  a morphism $f_0\to f_1$ in $\CAlg(\F)$ and $e_1\in\CAlg(\E_{f_1})\simeq\CAlg(\E)_{f_1}$, there are two ways to define a commutative algebra $e_0\in\CAlg(\E_{f_0})\simeq \CAlg(\E)_{f_0}$ together with a map $e_0\to e_1$ lying over $f_0\to f_1$, namely :
\begin{enumerate}
    \item Use the fact that $\CAlg(\E)\to\CAlg(\F)$ is a cartesian fibration and apply the pullback functor $\CAlg(\F)_{f_1}\to \CAlg(\F)_{f_0}$ to $e_1$
    \item Use the fact that the pushforward functor $\E_{f_0}\to \E_{f_1}$ is a symmetric monoidal left adjoint, so it has a lax symmetric monoidal right adjoint and this sends $e_1$ to some commutative algebra $e_0$.
\end{enumerate}
These two constructions yield the same $e_0\in\CAlg(\E_{f_0})$. 
\end{lm}
\begin{rmk}
If we could prove some analogous lemma for morphisms of ``pseudo-commutative algebras'' in a symmetric monoidal $(\infty,2)$-category (using the terminology of \cite{MV} again), then we could most likely conclude the comparison that we want. 

Indeed, to conclude the above proof of \Cref{cor : comp}, we apply this lemma two universes up to the bicartesian fibration $\widehat{\coCart}\to \widehat{\Cat}$ and the map of algebras $\C\to \Cat$ : this is where we need it to be strong symmetric monoidal : if it is only lax, it is not a morphism of commutative algebras in $\widehat{\Cat}$, but something like a morphism of ``pseudo-commutative algebras''.  
\pend\end{rmk}
\begin{proof}[Proof of \Cref{lm B}]
Here is a reformulation of the statement: $\E\to \F$ induces a lax symmetric monoidal functor $\F\to \Cat$, so we get a functor $\CAlg(\F)\to\CAlg(\Cat)$. We can further postcompose with $\CAlg(\Cat)\overset{\CAlg(-)}\to\Cat$ and observe that our hypotheses actually make this factor through $\mathrm{Adj}^L$, the category of categories and left adjoint functors. 

One then uses the equivalence $\mathrm{Adj}^L\simeq (\mathrm{Adj}^R)\op$ and forgets back down to $\Cat$ to get a functor $\CAlg(\F)\op\to \Cat$, which is classified by a cartesian fibration $\cat P\to\CAlg(\F)$. 

To prove the lemma, it is sufficient to prove that $\cat P\simeq\CAlg(\E)$ as cartesian fibrations over $\CAlg(\F)$. 

Because both are also coCartesian fibrations, it suffices to prove it for the corresponding coCartesian fibration. For $\CAlg(\E)$, we don't have anything to change, but for $\cat P$, we can simplify the above construction and not use the $\mathrm{Adj}^L\simeq (\mathrm{Adj}^R)\op$ equivalence. 

In other words, it suffices to prove that we have a pullback square $$\xymatrix{\CAlg(\E)\ar[d]\ar[rr] & & \Cat_{*\sslash}\ar[d] \\ 
\CAlg(\F) \ar[r] & \CAlg(\Cat) \ar[r]^-{\CAlg(-)} & \Cat}$$ 

because there is one for $\cat P$. 

But by \Cref{lm A}, the fiber is the correct one, so in fact it suffices to find a suitable commutative square. 

Because we have a commutative square of lax symmetric monoidal functors of the form~:  $$\xymatrix{\E \ar[r] \ar[d] & \Cat_{*\sslash}\ar[d]\\ \F\ar[r] & \Cat}$$ by applying $\CAlg(-)$ to it, we see that it suffices to exhbit a suitable commutative square of the form $$\xymatrix{\CAlg(\Cat_{*\sslash})\ar[r]\ar[d]& \Cat_{*\sslash}\ar[d]\\ \CAlg(\Cat)\ar[r]& \Cat}$$
i.e. the previous kind of diagram but in the universal case (note that the bottom map is not the forgetful functor, but the functor $\C\mapsto\CAlg(\C)$).

But this is exactly \Cref{lm : conditional}, i.e. the lemma that we did not prove. This is why this statement is conditional.

\end{proof}

In the general case where $\C\to\Cat$ is only lax symmetric monoidal, a similar approach could work if we managed to make the symmetric monoidal equivalence of lax symmetric monoidal functors $\C\mapsto \Fun(\C,\Cat), \C\mapsto\coCart_\C$ a symmetric monoidal equivalence of lax symmetric monoidal $(\infty,2)$-functors $\Cat^{co}\to \widehat{\Cat}$: indeed both the domain and the target are $(\infty,2)$-categories, and while commutative algebras in the $(\infty,1)$-category $\Cat$ only correspond to symmetric monoidal categories and strong symmetric monoidal functors, a relaxed notion of commutative algebra morphisms in the cartesian monoidal $(\infty,2)$-category $\Cat$ would encompass lax symmetric monoidal functors - in \cite{MV}, this is what is called a morphism of \emph{pseudomonoids} in a symmetric monoidal $2$-category. 
\begin{warn}
The $^{co}$ appearing up there adds an extra layer of difficulty. It stems from the fact that $\Fun(\C,\D)\op\to \Fun(\Fun(\C,\Cat),\Fun(\D,\Cat)), f\mapsto f_!$ is \emph{contravariant} (because $\Fun(\C\op,\D\op)\simeq \Fun(\C,\D)\op$, not $\Fun(\C,\D)$). In particular, a lax symmetric monoidal functor $\C\to \D$ induces an \emph{oplax} symmetric monoidal functor $\Fun(\C,\Cat)\to \Fun(\D,\Cat)$, while restriction is lax symmetric monoidal. 

One might think that this difficulty is artificial, as it goes away when looking at cartesian fibrations and contravariant functor $\C\op\to \Cat$, but the issue with this then becomes that a lax symmetric monoidal functor $\C\to \D$ is oplax symmetric monoidal as a functor $\C\op\to \D\op$, so this difficulty is here regardless. 
\pend\end{warn}

In fact, this kind of $(\infty,2)$-categorical technology and $(\infty,2)$-categorical analogues to the work we did here would be interesting in their own right, but they also seem needed to prove this property of the $(\infty,1)$-categorical version (at least at first sight); and in fact, this seems to be the only obstruction: an elaboration on the strategy outlined above shows that such results would imply the desired comparison. 

Note that \cite{MV} does deal with some $2$-categorical versions of these statements, so one should expect them to hold for $(\infty,2)$-categories as well. 
\appendix

\section{Colimits of coCartesian fibrations}\label{section : straight}
The goal of this appendix is to prove the following statement which featured in earlier approaches to our construction, and which is possibly of independent interest:
\begin{prop}\label{prop : colim}
The (non-full) subcategory inclusion $\coCart_\C\subset \Cat_{/\C}$ preserves colimits, and the subcategory is closed under tensoring with any $K\in\Cat$. 
\end{prop}

The key result we will use in the proof is \cite[Proposition 2.1.4]{hinich}, so for the convenience of the reader we spell it out:
\begin{prop}{\cite[Definition 2.1.1 and Proposition 2.1.4]{hinich}}
Let $q: \E\to \C$ be a functor, and $W\subset \E, V\subset \C$, classes of weak equivalences such that $q(W)\subset V$. For the induced functor $\E[W^{-1}] \to \C[V^{-1}]$ to be a coCartesian fibration, it suffices: 
\begin{enumerate}
    \item That $q$ be a coCartesian fibration; 
    \item That $q$-coCartesian edges lifing edges in $V$ be in $W$; 
    \item That for any morphism $f: c\to c'$ in $\C$, the induced functor $f_! : \E_c\to \E_{c'}$ send $W\cap \E_c$ to $W\cap \E_{c'}$;
    \item That for any morphism $f: c\to c'$ in $V$, the induced functor $\E_c[(W\cap E_c)^{-1}]\to~\E_{c'}[(W\cap E_{c'})^{-1}]$ be an equivalence
\end{enumerate}
\end{prop}
\begin{proof}[Proof of \Cref{prop : colim}]
The part about being closed under tensor with $K\in\Cat$ is the observation that the tensor $K\otimes (\D\to \C)$ is given by the composite $K\times\D\to K\times\C\to \C$, both of which are coCartesian fibrations.

For the claim about colimits, because $\coCart_\C \subset \Cat_{/\C}$ is a (non-full) subcategory, it suffices to prove that given a colimit diagram $\overline f : I^\triangleright \to \Cat_{/\C}$ such that its restriction $f$ to $I$ lands in $\coCart_\C$, the following three things are satisfied: the cocone point is also in $\coCart_\C$; the canonical maps to the cocone point are in $\coCart_\C$ as well; and finally given any cocone $f\to \D$ in $\coCart_\C$, that the induced map $\overline f(\infty)\to \D$ also lies in $\coCart_\C$ (where we use $\infty$ to denote the cocone point in $I^\triangleright$).

To adress the first of the three, let $f: I\to \coCart_\C$ be the restriction of our colimit cocone, and $(\mathbf P\to \C) = \overline f(\infty)$ its colimit in $\Cat_{/\C}$. The forgetful functor $\Cat_{/\C}\to\Cat$ preserves all colimits, so $\mathbf P$ is computed as a colimit of categories, namely by taking the cocartesian fibration $p: \int f\to I$ classifying $f$, and inverting the $p$-coCartesian edges. We let $W\subset \int f$ denote the class of $p$-coCartesian edges. 

Then the composite map $\int f \to \mathbf P\to \C$ is given by $\int f\to \C\times I\to \C$, where $\int f\to \C\times I$ is the morphism of coCartesian fibrations over $I$ determined by the natural transformation $f\to \C$ - here we abuse notation, and denote by $f$ the composite functor $I\to \Cat_{/\C}\to\Cat$, and by $\C$ the constant functor on $I$ with value $\C$.

By \Cref{lm : cartfiber}, $\int f\to \C\times I$ is a coCartesian fibration. It follows that $\int f\to \C\times I\to \C$ is one as well : we now want to apply \cite[Proposition 2.1.4]{hinich} to get that $\mathbf P=  \int f[W^{-1}]\to \C$ is also a coCartesian fibration. 

By this result it suffices to prove that $(\int f, W)\to (\C, \mathrm{equivalences})$ is a marked coCartesian fibration, following \cite[Definition 2.1.1]{hinich} - these are the four items we listed before this proof. The first item is that $\int f\to \C$ be a coCartesian fibration, which we just explained. The second item is obvious as the only marked arrows in $\C$ are equivalences; the third item follows from the fact that the diagram $f$ had values in $\coCart_\C$, so the morphisms $f(i)\to f(j)$ are morphisms of coCartesian fibrations over $\C$.
Finally, the fourth item also follows from the fact that the only marked arrows in $\C$ are equivalences. 

This proves that the colimit $\mathbf P\to \C$ is still a coCartesian fibration. 

We now need to show that the canonical morphisms $f(i)\to \mathbf P$ over $\C$ are in fact morphisms of coCartesian fibrations. 

This follows from the fact that we can write this as the composite of $$f(i)\to \int f \to \int f[W^{-1}].$$ By \cite[2.1.4]{hinich} again (and its proof), $\int f \to \int f[W^{-1}]$ corresponds to the natural transformation $(\int f)_d \to (\int f)_d[W_d^{-1}]$, and so is a morphism of coCartesian fibrations over $\C$, so it suffices to prove the same for $f(i)\to \int f$. But this follows from the proof of \Cref{lm : cartfiber}, where we see that the inclusion $X_s\to X$ sends $p_s$-coCartesian edges to $p$-coCartesian edges - and in $\int f$, a coCartesian edge over $\C\times I$ that furthermore lives over an identity in $I$ is cocartesian over $\C$.

Finally, we need to argue that if $f\to \D$ is a cocone in $\coCart_\C$, then the induced map $\mathbf P\to \D$ is also in $\coCart_\C$, i.e. it preserves coCartesian edges.  

The point is that while some edges in $\mathbf P$ do not lift to any $f(i)$ (there can be some zigzags) all coCartesian edges in $\mathbf P$ \emph{do} come from some $f(i)$. Indeed, consider an edge $e: x\to y$ in $\C$, together with a lift $\tilde x$ of $x$ to $\mathbf P$. As $\mathbf P$ is a localization of $\int f$, $\tilde x$ is the image of some $\overline x\in f(i)$, for some $i\in I$, and because $f(i)\to \mathbf P$ is over $\C$, $\overline x$ is a lift of $x$. Therefore we can find a coCartesian lift $\overline x\to e_!\overline x$ of $e$ in $f(i)$. Now we proved above that  $f(i)\to \mathbf P$ preserves coCartesian edges, so the image of $\overline x\to e_! \overline x$ is coCartesian in $\mathbf P$, starts at $\tilde x$ and lifts $e$. Any coCartesian edge with these properties will therefore be equivalent to this one. 

So now, if $\alpha$ is a coCartesian edge in $\mathbf P$ coming from $f(i)$, because $f(i)\to \D$ preserves coCartesian edges, the image of $\alpha$ in $\D$ is coCartesian : $\mathbf P\to\D$ preserves coCartesian edges. 
\end{proof}
\begin{rmk}
It follows formally that $\coCart_\C \subset\Cat_{/\C}$ is also closed under partially lax colimits, cf. \cite[Definition 4.11]{partlax}. Alternatively, by \cite[Theorem 4.13(b)]{partlax} these partially lax colimits can be described in terms of a certain localization of the Grothendieck construction, and the proof above applies verbatim to this more general case. 
\pend\end{rmk}
\begin{rmk}
  The inclusion $\coCart_\C\subset \Cat_{/\C}$ has a left adjoint \cite[Theorem 1.2]{gepnerhaugsengnik}, so it also preserves limits. This is more expected, as coCartesian fibrations are defined by certain right lifting properties.
  
  It would be interesting to describe a right adjoint similarly to what is done in \cite{gepnerhaugsengnik}. 
\pend\end{rmk}
We also obtain analogous results for cartesian fibrations, because $\C\mapsto \C\op$ is a self-equivalence of $\Cat$ which interchanges coCartesian fibrations and cartesian fibrations; and also the same results for left and right fibrations because $\Ss\subset \Cat$ is closed under colimits. We record this as: 
\begin{cor}
Let $\C$ be a category. The (non-full) subcategory inclusions $$\LFib_\C, \RFib_\C, \coCart_\C, \Cart_\C \subset \Cat_{/\C}$$ of respectively left fibrations, right fibrations, coCartesian fibrations and cartesian fibrations, preserve colimits. 

In the case of co/Cartesian fibrations, they are also stable under tensoring with arbitrary categories, and are thus also stable under arbitrary partially lax colimits. 
\end{cor}

Because $\coCart_\C\subset \Cat_{/\C}$ is a (non-full) subcategory stable under tensoring with any $K\in \Cat$, it acquires a canonical $\Cat$-linear structure compatible with the inclusion. 

Similarly, $\Fun(\C, \Cat)$ has a canonical $\Cat$-linear structure. A folklore fact, which will be important below, is that the un/straightening equivalence is $\Cat$-linear. We prove this below.

\begin{lm}\label{lm : catlinear}
For any category $\C$, the un/straightening-equivalence $\coCart_\C\simeq \Fun(\C,\Cat)$ is canonically $\Cat$-linear. 
\end{lm}
\begin{proof}
Because the un/straightening equivalence is natural in $\C$, see \cite[Corollary A.32]{gepnerhaugsengnik}, and the terminal category is simply a point $*$, we have a commutative square of the form
\[\begin{tikzcd}
	{\coCart_\C} & {\Fun(\C,\Cat)} \\
	{\coCart_*} & {\Fun(*,\Cat)}
	\arrow["\simeq", from=1-1, to=1-2]
	\arrow[from=1-1, to=2-1]
	\arrow[from=1-2, to=2-2]
	\arrow["\simeq", from=2-1, to=2-2]
\end{tikzcd}\]

The horizontal functor is an equivalence, so we can take right adjoints vertically, and thus get a commutative square
\[\begin{tikzcd}
	{\coCart_\C} & {\Fun(\C,\Cat)} \\
	{\coCart_*} & {\Fun(*,\Cat)}
	\arrow["\simeq", from=1-1, to=1-2]
	\arrow[from=2-1, to=1-1]
	\arrow[from=2-2, to=1-2]
	\arrow["\simeq", from=2-1, to=2-2]
\end{tikzcd}\]
Furthermore, all the functors involved in this natural square are product-preserving (the horizontal ones are equivalences, and the vertical ones are right adjoints), so that this can be seen as a square in $\CAlg(\widehat{\Cat})$ where we give every category the \emph{cartesian} monoidal structure\footnote{This is unrelated to the monoidal Grothendieck construction which is the focus of this paper -  this is the much simpler observation that for functors between cartesian monoidal categories, symmetric monoidality is the \emph{property} of preserving products. }.

In particular, the equivalence $\coCart_\C \simeq \Fun(\C,\Cat)$ can be viewed as an equivalence in $\CAlg(\widehat{\Cat})_{\Cat/}$, but the latter is equivalent to $\CAlg(\Mod_\Cat(\widehat{\Cat}))$ (see \cite[Corollary 3.4.1.7]{HA} and \cite[Corollary 4.5.1.6]{HA}) so has a forgetful functor to $\Mod_\Cat(\widehat{\Cat})$. 

Finally, we note that the module structures that we obtain this way are the natural ones.
\end{proof}
\begin{rmk}
With a bit more work about ``passing to right adjoints'', we could make this $\Cat$-linear structure natural in $\C$. However, it would take more time, and we do not really need that much detail, so we do not wish to linger on naturality here.  
\pend\end{rmk}
We obtain:
\begin{cor}\label{cor : unstraightisextended}
Let $\C$ be a category. The unstraightening functor $\Fun(\C,\Cat)\to \coCart_\C$ is the unique colimit-preserving, $\Cat$-linear extension of its restriction along the Yoneda embedding $\C\op\to \Fun(\C,\Ss)\to \Fun(\C,\Cat)$ - this restriction is given by $x\mapsto (\C_{x/}\to \C)$. 

The same holds for the composition of the unstraightening functor with the forgetful functor, $\Fun(\C,\Cat)\to \coCart_\C\subset \Cat_{/\C}$.  
\end{cor}
\begin{proof}
The uniqueness of such an extension follows from the equivalence $\Fun(\C,\Cat)\simeq\Fun(\C,\Ss)\otimes\Cat$ and the universal property of presheaves - here, $\otimes$ is the Lurie tensor product of presentable categories (see \cite[Section 4.8.1]{HA}). 

For the first part, the un/straightening functor is an equivalence so it obviously preserves colimits, and \Cref{lm : catlinear} proves that it is $\Cat$-linear. 

For the second part of the statement, the same thing works except that now we use \Cref{prop : colim} to prove that colimit-preservation and $\Cat$-linearity are preserved under the forgetful functor to $\Cat_{/\C}$. 
\end{proof}

\bibliographystyle{alpha}
\bibliography{Biblio.bib}

@misc{HA,
  title={Higher algebra},
  author={Lurie, Jacob},
  year={2012}
}

@book{HTT,
  title={Higher topos theory},
  author={Lurie, Jacob},
  year={2009},
  publisher={Princeton University Press}
}

@article{NS,
  title={On topological cyclic homology},
  author={Nikolaus, Thomas and Scholze, Peter},
  journal={Acta Mathematica},
  volume={221},
  number={2},
  pages={203--409},
  year={2018},
  publisher={Institut Mittag-Leffler}
}

@article{MV,
  title={Monoidal {G}rothendieck construction},
  author={Moeller, Joe and Vasilakopoulou, Christina},
  journal={Theory and Applications of Categories, Vol. 35, 2020, No. 31, pp 1159-1207},
  year={}
}

@book{land,
  title={Introduction to Infinity-Categories},
  author={Land, Markus},
  year={2021},
  publisher={Springer Nature}
}

@article{thom,
  title={A simple universal property of {T}hom ring spectra},
  author={Antol{\'\i}n-Camarena, Omar and Barthel, Tobias},
  journal={Journal of Topology},
  volume={12},
  number={1},
  pages={56--78},
  year={2019},
  publisher={Wiley Online Library}
}

@article{hinich,
  title={Dwyer-{K}an localization revisited},
  author={Hinich, Vladimir},
  journal={  Homology, Homotopy, appl., 18 (2016), 27-48},
  year={2016}
}

@article{gepnerhaugsengnik,
  title={Lax colimits and free fibrations in {$\infty $}-categories},
  author={Gepner, David and Haugseng, Rune and Nikolaus, Thomas},
  journal={Documenta Mathematica 22 (2017) 1225-1266},
  year={2017}
}

@article{hinichrect,
  title={Rectification of algebras and modules},
  author={Hinich, Vladimir},
  journal={Doc. Math., 20 (2015), 879-926},
  year={2015}
}

@article{baezdolan,
  title={Higher-dimensional algebra {III}. n-categories and the algebra of opetopes},
  author={Baez, John C and Dolan, James},
  journal={Advances in Mathematics},
  volume={135},
  number={2},
  pages={145--206},
  year={1998},
  publisher={Elsevier}
}

@misc{nLabmicro,
author = {n{L}ab authors},
title = {Microcosm principle},
year = {2022},
howpublished = "\url{https://ncatlab.org/nlab/show/microcosm+principle}"
}

@article{haugchu,
  title={Free algebras through {D}ay convolution},
  author={Chu, Hongyi and Haugseng, Rune},
  journal={Algebraic \& Geometric Topology},
  volume={22},
  number={7},
  pages={3401--3458},
  year={2023},
  publisher={Mathematical Sciences Publishers}
}

@article{BSP,
  title={On the unicity of the theory of higher categories},
  author={Barwick, Clark and Schommer-Pries, Christopher},
  journal={Journal of the American Mathematical Society},
  volume={34},
  number={4},
  pages={1011--1058},
  year={2021}
}

@article{HR,
  title={A multiplicative comparison of {M}ac {L}ane homology and topological {H}ochschild homology},
  author={Horel, Geoffroy and Ramzi, Maxime},
  journal={Annals of K-Theory},
  volume={6},
  number={3},
  pages={571--605},
  year={2021},
  publisher={Mathematical Sciences Publishers}
}

@phdthesis{heine,
  title={Restricted {$L_\infty$}-algebras},
  author={Heine, Hadrian},
  year={2019},
  school={Universit{\"a}t Osnabr{\"u}ck}
}

@article{HHLN,
title={Lax monoidal adjunctions, two-variable fibrations and the calculus of mates},
  author={Haugseng, Rune and Hebestreit, Fabian and Linskens, Sil and Nuiten, Joost},
  journal={Proceedings of the London Mathematical Society},
  volume={127},
  number={4},
  pages={889--957},
  year={2023},
  publisher={Wiley Online Library}
}

@article{haugseng2022shifted,
  title={Shifted coisotropic correspondences},
  author={Haugseng, Rune and Melani, Valerio and Safronov, Pavel},
  journal={Journal of the Institute of Mathematics of Jussieu},
  volume={21},
  number={3},
  pages={785--849},
  year={2022},
  publisher={Cambridge University Press}
}

@article{HHLN2, 
  title={Two-variable fibrations, factorisation systems and-categories of spans},
  author={Haugseng, Rune and Hebestreit, Fabian and Linskens, Sil and Nuiten, Joost},
  journal={Forum of Mathematics, Sigma},
  volume={11},
  year={2023},
  organization={Cambridge University Press}
  
}

@article{HLBrauer,
  title={On {B}rauer {G}roups of {L}ubin-{T}ate {S}pectra {I}},
  author={Hopkins, Michael J and Lurie, Jacob},
  journal={preprint available at http://www. math. harvard. edu/\~{} lurie},
  year={2017}
}

@article{partlax,
   title={Global homotopy theory via partially lax limits},
  author={Linskens, Sil and Nardin, Denis and Pol, Luca},
  journal={Geometry \& Topology},
  volume={29},
  number={3},
  pages={1345--1440},
  year={2025},
  publisher={Mathematical Sciences Publishers}
}

@article{yoneda,
  title={An elementary proof of the naturality of the Yoneda embedding},
  author={Ramzi, Maxime},
  journal={Proceedings of the American Mathematical Society},
  volume={151},
  number={10},
  pages={4163--4171},
  year={2023}
}
\textsc{Institut for Matematiske Fag, K\o benhavns Universitet, Danmark}

\textit{Email adress : }\texttt{maxime.ramzi@math.ku.dk}
\end{document}